%% file: k4minorfreeFINAL.tex
\documentclass[a4paper,10pt]{article}

\usepackage[utf8]{inputenc}
\usepackage[english]{babel}
\usepackage{hyperref}
\usepackage{amsmath,amsthm,enumerate}
\usepackage{amssymb}
\usepackage{graphicx}
\usepackage{tikz}
\usetikzlibrary{calc,decorations.pathmorphing,decorations.text,fixedpointarithmetic}
\usepackage{subfigure}
\usepackage{refcount}
\usepackage[hmargin=2.5cm,vmargin=3cm]{geometry}
\usepackage{algpseudocode,algorithm,algorithmicx}
\algrenewcommand\algorithmicrequire{\textbf{Input:}}
\algrenewcommand{\algorithmiccomment}[1]{\hfill{\color{gray}\# #1}}
\newcommand*\Let[2]{\State #1 $\gets$ #2}

\newtheorem{theorem}{Theorem}

\newtheorem{proposition}[theorem]{Proposition}
\newtheorem{lemma}[theorem]{Lemma}
\newtheorem{observation}[theorem]{Observation}
\newtheorem{corollary}[theorem]{Corollary}

\newtheorem{definition}[theorem]{Definition}

\newtheorem{conjecture}[theorem]{Conjecture}
\newtheorem{claim}{Claim}[theorem]

\newcommand{\smallqed}{\hfill{\tiny $\left(\Box\right)$}}
\newcommand{\claimproof}{\noindent\emph{Proof of claim.} }

\newtheorem{problem}[theorem]{Problem}

\newcommand{\SPG}{\mathcal{SP}}
\newcommand{\SP}[1]{\mathcal{SP}_{#1}}

\usepackage[color=green!30]{todonotes}

\begin{document}

\title{Homomorphism bounds and edge-colourings\\of $K_4$-minor-free graphs}
\author{Laurent Beaudou\footnote{\noindent LIMOS - CNRS UMR 6158, Universit\'e Clermont Auvergne, Clermont-Ferrand (France). E-mails: \{laurent.beaudou,florent.foucaud\}@isima.fr}
\and Florent Foucaud\footnotemark[1]
\and Reza Naserasr\footnote{\noindent CNRS - LIAFA UMR 7089, Université Paris Diderot - Paris 7 (France). E-mail: reza@lri.fr}}

\maketitle




\begin{abstract}

We present a necessary and sufficient condition for a graph of
odd-girth $2k+1$ to bound the class of $K_4$-minor-free graphs of
odd-girth (at least) $2k+1$, that is, to admit a homomorphism from any
such $K_4$-minor-free graph. This yields a polynomial-time algorithm
to recognize such bounds. Using this condition, we first prove that
every $K_4$-minor free graph of odd-girth $2k+1$ admits a homomorphism
to the projective hypercube of dimension $2k$. This supports a conjecture
of the third author which generalizes the four-color theorem and
relates to several outstanding conjectures such as Seymour's
conjecture on edge-colorings of planar graphs. Strengthening this
result, we show that the Kneser graph $K(2k+1,k)$ satisfies the
conditions, thus implying that every $K_4$-minor free graph of
odd-girth $2k+1$ has fractional chromatic number exactly
$2+\frac{1}{k}$. Knowing that a smallest bound of odd-girth $2k+1$
must have at least ${k+2 \choose 2}$ vertices, we build nearly
optimal bounds of order $4k^2$. Furthermore, we conjecture that the
suprema of the fractional and circular chromatic numbers for
$K_4$-minor-free graphs of odd-girth $2k+1$ are achieved by a same
bound of odd-girth $2k+1$. If true, this improves, in the homomorphism
order, earlier tight results on the circular chromatic number of
$K_4$-minor-free graphs.  We support our conjecture by proving it for
the first few cases. Finally, as an application of our work, and after
noting that Seymour provided a formula for calculating the edge-chromatic
number of $K_4$-minor-free multigraphs, we show that stronger results
can be obtained in the case of $K_4$-minor-free regular multigraphs.
\end{abstract}

\section{Introduction}

In this paper, graphs are simple and loopless. While loops will never be considered, multiple edges 
will be considered, but graphs containing multiple edges will specifically be referred to as 
multigraphs. 
A homomorphism $f$ of a graph $G$ to a graph $H$ is a mapping of $V(G)$ to $V(H)$ that preserves 
the edges, that is, if $x$ and $y$ are adjacent in $G$, then $f(x)$ and $f(y)$ are adjacent in $H$. 
If such a homomorphism exists, we note $G\to H$. Observe that a homomorphism of a graph $G$ to the 
complete graph $K_k$ is the same as a proper $k$-colouring of $G$. Given a class $\mathcal C$ of 
graphs and a graph $H$, we say that $H$ \emph{bounds} $\mathcal{C}$ if every graph in $\mathcal{C}$ 
admits a homomorphism to $H$. We also say that $\mathcal C$ is \emph{bounded} by $H$, or that $H$ is 
a \emph{bound} for $\mathcal C$. For example, the Four-Colour Theorem states that $K_4$ bounds the 
class of planar graphs. It is of interest to also ask for (small) bounds having some additional 
properties. In this paper we study the problem of determining bound(s) of smallest possible order 
for the class of $K_4$-minor-free graphs of given odd-girth, with the restriction that the bound 
itself also has the same odd-girth (the \emph{odd-girth} of a graph is the smallest length of an odd cycle).

\paragraph{Homomorphism bounds and minors} The \emph{core} of a graph $G$ is the smallest subgraph of $G$ to which $G$ admits a homomorphism (it is known to be unique up to isomorphism). A graph $G$ is a \emph{core} if it is its own core. For additional concepts in graph homomorphisms, we refer to the book by Hell and Ne\v{s}et\v{r}il~\cite{HNbook}. Let $\mathcal I=\{I_1,I_2, \ldots, I_i\}$ be a class of cores. We define $forb_h(\mathcal I)$ to be the class of all graphs to which no member of $\mathcal I$ admits a homomorphism. For example for $\mathcal{I}=\{K_n\}$, $forb_h(\mathcal I)$ is the class of $K_n$-free graphs and for $\mathcal{I}=\{C_{2k-1}\}$ it is the class of graphs of odd-girth at least $2k+1$. Similarly, given $\mathcal J=\{J_1,J_2, \ldots, J_j\}$, we use the notation $forb_m(\mathcal J)$ to denote the class of all graphs that have no member of $\mathcal J$ as a minor. (Recall that a graph $H$ is a \emph{minor} of a graph $G$ if $H$ can be obtained from $G$ by a sequence of edge-contractions and vertex- and edge-deletions.) The following is a fundamental theorem in the study of the relation between minors and homomorphisms.

\begin{theorem}[Ne\v{s}et\v{r}il and Ossona de Mendez \cite{NO08}]\label{thm:nes-OdM} Given a finite set $\mathcal I$ of connected graphs and a finite set $\mathcal J$ of graphs, there is a graph in $forb_h(\mathcal I)$ which is a bound for the class $forb_h(\mathcal I) \cap forb_m(\mathcal J)$.
\end{theorem}

It is worthwhile to note that a bound for $forb_h(\mathcal I) \cap
forb_m(\mathcal J)$ belonging itself to $forb_h(\mathcal I) \cap
forb_m(\mathcal J)$ may not exist, for example there is no
triangle-free planar graph bounding the class of triangle-free planar
graphs~\cite{MNN06} (see Theorem~3.37 in the third author's PhD
thesis~\cite{RezaThesis} for a more general statement for any
odd-girth, and~\cite{NN12} for a further generalization). By a similar argument, it is easily observed that there is
no $K_4$-minor-free graph of odd-girth $2k+1$~($k \geq 2$) bounding
the class of $K_4$-minor-free graphs of odd-girth $2k+1$.

Even though the proof of Theorem~\ref{thm:nes-OdM} is constructive,
the bound obtained by this construction is generally very large. Thus,
we pose the following problem.

\begin{problem}\label{optimal bound}
What is a graph in $forb_h(\mathcal I)$ of smallest possible order that bounds 
$forb_h(\mathcal I) \cap forb_m(\mathcal J)$?
\end{problem}

Finding the answer to Problem~\ref{optimal bound} can be a very difficult task in general. For example, Hadwiger's conjecture would be answered if we find the optimal solution with $\mathcal I =\{K_n\}$ and $\mathcal J =\{K_n\}$. Nevertheless, the answer to some special cases is known. As an example, the main result of~\cite{G88} implies in particular the following claim, which corresponds to $\mathcal I =\{C_{2k-1}\}$ and $\mathcal J =\{K_4, K_{2,3}\}$.

\begin{theorem}[Gerards \cite{G88}]\label{Gerard}
The cycle $C_{2k+1}$ bounds the class of outerplanar graphs of odd-girth at least~$2k+1$.
\end{theorem}

Problem~\ref{optimal bound} asks for optimal bounds in terms of the order. As observed 
in~\cite{NO05,NO08}, such a bound (unless it is itself in the class $forb_m(\mathcal J)$) cannot be 
optimal with respect to the homomorphism order. More precisely, if $B\in forb_h(\mathcal I)$ is a 
bound for $forb_h(\mathcal I) \cap forb_m(\mathcal J)$ and $B\notin forb_m(\mathcal J)$, then there 
exists another bound $B'\in forb_h(\mathcal I)$ such that $B'\to B$ and $B \nrightarrow B'$.

One of the interesting cases of Problem~\ref{optimal bound} is when
$\mathcal I=\{C_{2k-1}\}$ and $\mathcal J=\{K_{3,3}, K_5\}$. In other
words, the problem consists of finding a smallest graph of odd-girth
$2k+1$ which bounds the class of planar graphs of odd-girth at least
$2k+1$. This particular problem was studied by several authors, see
for example~\cite{G05,MNN06,N07,N13}. A proposed answer was formulated
as a conjecture, using the notions that are developed next.

\paragraph{Projective hypercubes and bounds for planar graphs}
Hypercubes are among the most celebrated families of graphs. Recall
that the \emph{hypercube} of dimension $d$, denoted $H(d)$, is the
graph whose vertices are all binary words of length $d$ and where two
such words are adjacent if their Hamming distance is 1. In other
words, $H(d)$ is the Cayley graph $(\mathbb{Z}_2^d, \{e_1,e_2, \ldots,
e_{d}\})$ where $\{e_1,e_2, \ldots, e_{d}\}$ is the standard
basis. With a natural embedding of the $d$-dimensional hypercube
$H(d)$ on the $d$-dimensional sphere $S^d$, hypercubes can be regarded
as the discrete approximation of $S^d$. Recall that the projective
space of dimension $d$ is obtained from identifying antipodal points
of $S^{d+1}$. Then, the image of $H(d+1)$ under such projection is
called the \emph{projective hypercube} of dimension $d$, denoted
$PC(d)$ (sometimes also called \emph{projective cube} or \emph{folded
  cube}). It is thus the graph obtained from the $(d+1)$-dimensional
hypercube by identifying all pairs of antipodal vertices. It can be
checked that this is the same as the graph obtained from the
$d$-dimensional hypercube by adding an edge between every pair of
antipodal vertices. Thus, $PC(d)$ is the Cayley graph
$(\mathbb{Z}_2^{d}, S=\{e_1,e_2, \ldots, e_{d},J\})$ where $\{e_1,e_2,
\ldots, e_{d}\}$ is the standard basis and $J$ is the all-$1$
vector. Therefore, $PC(1)$ is $K_2$, $PC(2)$ is $K_4$, $PC(3)$ is
$K_{4,4}$ and $PC(4)$ is the celebrated Clebsch graph (which is also
referred to as the Greenwood-Gleason graph for its independent
appearance in Ramsey theory).

As a generalization of the Four-Colour Theorem, the following conjecture was proposed by the third author.

\begin{conjecture}[Naserasr \cite{N07}]\label{PlanarToPC}
The projective hypercube $PC(2k)$ bounds the class of planar graphs of odd-girth at least $2k+1$.
\end{conjecture}

Since $PC(2)$ is isomorphic to $K_4$, the first case of this conjecture is the Four-Colour Theorem. 
The conjecture is related to determining the edge-chromatic number of a class of planar multigraphs, 
as we will explain later. Moreover, it is shown by Naserasr, Sen and Sun~\cite{NSS15} that no bound 
of odd-girth $2k+1$ of smaller order can exist. Thus this conjecture proposes a solution to 
Problem~\ref{optimal bound} for the case $\mathcal I=\{C_{2k-1}\}$ and 
$\mathcal J=\{K_{3,3}, K_5\}$.

\paragraph{Edge-colouring and fractional colouring}
Recall that colouring 
a graph $G$ corresponds to partitioning its vertices into independent sets, and 
the chromatic number of $G$, denoted $\chi(G)$, is the smallest size of such a 
partition. Analogously, edge-colouring a multigraph $G$ corresponds to 
partitioning its edge set into matchings, and the edge-chromatic number of $G$, 
denoted $\chi'(G)$, is the smallest size of such a partition. 
Each of these parameters can be regarded as an optimal solution 
to an integer program where independent sets (or matchings) are given
value $0$ or $1$ together with inequalities indicating that each vertex
(or edge) receives a total weight of at least~$1$, that is, 
it is coloured (or covered). 

After relaxing these possible values from $\{0, 1\}$ to all
non-negative real numbers, we obtain the notion of fractional
chromatic number, denoted $\chi_f(G)$, and fractional edge-chromatic
number of multigraphs, denoted $\chi_f'(G)$. The fractional chromatic
number of a graph $G$ can equivalently be defined as the smallest
value of $\frac{p}{q}$ over all positive integers $p,q$, $2q\leq p$,
such that $G \to K(p,q)$, where $K(p,q)$ is the \emph{Kneser graph}
whose vertices are all $q$-subsets of a $p$-set and where two vertices
are adjacent if they have no element in common.

While computing each of $\chi(G)$, $\chi_f(G)$ and $\chi'(G)$ are known to be NP-complete problems (see~\cite{K72}, \cite{LY94} and~\cite{H81} respectively), it is shown by Edmonds~\cite{E65} (see also~\cite{S79}) that $\chi'_f(G)$ can be computed in polynomial time. We refer to the book of Scheinermann and Ullman~\cite{SU97} for details.
For every multigraph $G$, the inequality $\chi_f'(G)\leq \chi'(G)$ clearly holds. Seymour conjectured that, up to an integer approximation, equality holds for every planar multigraph.

\begin{conjecture}[Seymour \cite{S75,S79a}]\label{EdgeColouringPlanars}
For each planar multigraph $G$, $\chi'(G)=\lceil\chi_f'(G)\rceil$.
\end{conjecture}

The restriction of Conjecture~\ref{EdgeColouringPlanars} to planar $3$-regular multigraphs corresponds to a claim of Tait (every bridgeless cubic planar graphs is 3-edge-colourable) from the late $19$th century~\cite{T80}. As Tait has shown, this is equivalent to the Four-Colour Theorem.

Conjecture~\ref{EdgeColouringPlanars} has been studied extensively for the special case of planar \emph{$r$-graphs}, for which the fractional edge-chromatic number is known to be exactly $r$~\cite{S75,S79}. An $r$-regular multigraph is an $r$-graph if for each set $X$ of an odd number of vertices, the number of edges leaving $X$ is at least $r$. Hence, Conjecture~\ref{EdgeColouringPlanars} restricted to this class is stated as follows.

\begin{conjecture}[Seymour \cite{S75,S79a}]\label{EdgeColouringRegularPlanars}
Every planar $r$-graph is $r$-edge-colourable.
\end{conjecture}

For any odd integer $r=2k+1$ ($k\geq 1$), the claim of Conjecture~\ref{EdgeColouringRegularPlanars} is proved to be equivalent to the claim of Conjecture~\ref{PlanarToPC} for $k$ by Naserasr~\cite{N07}. Conjecture~\ref{EdgeColouringRegularPlanars} has been proved for $r\leq 8$ in~\cite{G12} ($r=4,5$), \cite{DKK10} ($r=6$), \cite{CEKS15,E11} ($r=7$) and \cite{CEKS12} ($r=8$). We note that these proofs use induction on $r$ and thus use the Four-Colour Theorem as a base step.

The claim of Conjecture~\ref{EdgeColouringPlanars} when restricted to the class of $K_4$-minor-free graphs (a subclass of planar graphs) was proved by Seymour~\cite{S90}. A simpler (unpublished) proof of this result is given more recently by Fernandes and Thomas~\cite{FT10}.

Using Seymour's result and a characterization theorem, Marcotte extended the result to the class of graphs with no $K_{3,3}$ or $(K_5-e)$-minor~\cite{M01}, where $K_5-e$ is obtained from $K_5$ by removing an edge. 

A key tool in proving the equivalence between Conjecture~\ref{PlanarToPC} and Conjecture~\ref{EdgeColouringPlanars} (for a fixed $k$) in~\cite{N07} is the Folding Lemma of Klostermeyer and Zhang~\cite{KZ00}. In the absence of such a lemma for the subclass of $K_4$-minor-free graphs, there is no known direct equivalence between the restrictions of Conjecture~\ref{EdgeColouringRegularPlanars} and Conjecture~\ref{PlanarToPC} to $K_4$-minor-free graphs. However, since $K_4$-minor-free graphs are planar, the notion of dual is well-defined. Using such a notion, Conjecture~\ref{EdgeColouringRegularPlanars} restricted to the class of $K_4$-minor-free graphs is implied by Conjecture~\ref{PlanarToPC} restricted to the same class. Thus, in this paper, by proving Conjecture~\ref{PlanarToPC} for the class of $K_4$-minor-free graphs, we obtain as a corollary a new proof of Conjecture~\ref{EdgeColouringRegularPlanars} for $K_4$-minor-free $(2k+1)$-graphs. Then, by improving on the homomorphism bound, we also deduce stronger results on the edge-colouring counterpart.

\paragraph{Our results} In this paper, we study the case $\mathcal I=\{C_{2k-1}\}$ and $\mathcal J=\{K_4\}$ of Problem~\ref{optimal bound}, that is, the case of $K_4$-minor-free graphs (also known as series-parallel graphs) of odd-girth at least~$2k+1$, that we denote by $\SP{2k+1}$. Our main tool is to prove necessary and sufficient conditions for a graph $B$ of odd-girth $2k+1$ to be a bound for $\SP{2k+1}$. These conditions are given in terms of the existence of a certain weighted graph (that we call a $k$-partial distance graph of $B$) containing $B$ as a subgraph, and that satisfies certain properties. The main idea of the proof is based on homomorphisms of weighted graphs and the characterization of $K_4$-minor-free graphs as partial $2$-trees. This result is presented in Section~\ref{sec:maintheorem}. From this, we are able to deduce a polynomial-time algorithm to decide whether a graph of odd-girth~$2k+1$ is a bound for $\SP{2k+1}$. This algorithm is presented in Section~\ref{sec:algo}. We will then use our main theorem, in Section~\ref{sec:general}, to prove that the projective hypercube $PC(2k)$ bounds $\SP{2k+1}$, showing that Conjecture~\ref{PlanarToPC} holds when restricted to $K_4$-minor-free graphs. In fact, we also show that this is far from being optimal (with respect to the order), by exhibiting two families of subgraphs of the projective hypercubes that are an answer: the Kneser graphs $K(2k+1,k)$, and a family of order $4k^2$ (which we call augmented square toroidal grids). Note that the order $O(k^2)$ is optimal, as shown by He, Sun and Naserasr~\cite{HNS15}, while for planar graphs it is known that any answer must have order at least $2^{2k}$ (see Sen, Sun and Naserasr~\cite{NSS15}). In Section~\ref{sec:smallvalues}, for $k\leq 3$, we determine optimal answers to the problem. For $k=1$, it is well-known that $K_3$ is a bound; $K_3$ being a $K_4$-minor-free graph, it is the optimal bound in many senses (in terms of order, size, and homomorphism order). We prove that the smallest triangle-free graph bounding $\SP{5}$ has order~$8$, and the smallest graph of odd-girth~$7$ bounding $\SP{7}$ has order~$15$ (and we determine concrete bounds of these orders). These graphs are not $K_4$-minor-free, and, therefore, these two bounds are not optimal in the sense of the homomorphism order. All our bounds are subgraphs of the corresponding projective hypercubes. These optimal bounds for $k\leq 3$ yield a strengthening of other results about both the fractional and the circular chromatic numbers of graphs in $\SP{5}$ and $\SP{7}$. In Section~\ref{sec:applis}, we discuss applications of our work to edge-colourings of $K_4$-minor-free multigraphs. We finally conclude with some remarks and open questions in Section~\ref{sec:remarks}.

\medskip
%

\section{Preliminaries}\label{sec:prelim}

In this section, we gather some definitions and useful results from the literature.

\subsection{General definitions and observations}

Given three positive real numbers $p,q,r$ we say the triple
$\{p,q,r\}$ satisfies the \emph{triangular inequalities} if we have
$p\leq r+q$, $q\leq p+r$ and $r\leq p+q$. Assuming $p$ is the largest
of the three, it is enough to check that $p\leq q+r$.

In a graph $G$, we denote by $N_G^d(v)$ the \emph{distance $d$-neighbourhood of 
$v$}, that is, the set of vertices at distance exactly~$d$ from vertex $v$; 
$N_G^1(v)$ is simply denoted by $N_G(v)$. In $G$, the distance between two 
vertices $u$ and $v$ is denoted $d_G(u,v)$. In these notations, if there is no ambiguity about the graph $G$, the subscript may be omitted.

An \emph{independent set} in a graph is a set of vertices no two of which are adjacent.

A \emph{walk} in $G$ is a sequence $v_0,\ldots, v_k$ of vertices where
two consecutive vertices are adjacent in $G$. A walk with $k$ edges is
a \emph{$k$-walk}. If the first and last vertices are the same, the
walk is a \emph{closed walk} (\emph{closed $k$-walk}). If no vertex of
a walk is repeated, then it is a \emph{path} ($k$-path). A path whose
internal vertices all have degree~$2$ is a \emph{thread}. If in a
closed walk, no inner-vertex is repeated, then it is a \emph{cycle}
($k$-cycle).

Given a set $X$ of vertices of a graph $G$, we denote by $G[X]$ the subgraph of 
$G$ induced by $X$.

Given two graphs $G$ and $H$, the cartesian product of $G$ and $H$, denoted 
$G \Box H$, is the graph on vertex set $V(G)\times V(H)$ where $(x,y)$ is 
adjacent 
to $(z,t)$ if either 
$x=z$ and $y$ is adjacent to $t$ in $H$, or $y=t$ and $x$ is adjacent to $z$ in 
$G$.

An \emph{edge-weighted graph} $(G,w)$ is a graph $G$ together with an 
edge-weight function $w:E(G)\to~\mathbb{N}$. Given two edge-weighted graphs 
$(G,w_1)$ and $(H,w_2)$, a homomorphism of $(G,w_1)$ to $(H,w_2)$ is a 
homomorphism of $G$ to $H$ which also preserves the edge-weights.
Given a connected graph $G$ of order $n$, the \emph{complete distance graph} $(K_n,d_G)$ 
of $G$ is the weighted complete graph on vertex set $V(G)$, and where for each 
edge $u,v$ of $K_n$, its weight $\omega(uv)$ is the distance $d_G(u,v)$ 
between $u$ and $v$ in $G$. Given any subgraph $H$ of $K_n$, the \emph{partial 
distance graph} $(H,d_G)$ of $G$ is the spanning subgraph of $(K_n,d_G)$ whose 
edges are the edges of $H$. Furthermore, if for every edge $xy$ of $H$ we have $d_G(x,y)\leq k$, we say that $(H,d_G)$ is a \emph{$k$-partial distance graph of $G$}.

Easy but important observations, which we will use frequently, 
are the following.

\begin{observation}\label{obs:DistanceDetmByCycle}
 Let $G$ be a graph of odd-girth $2k+1$ and $C$ a cycle of length
 $2k+1$ in $G$. Then, for any pair $(u,v)$ of vertices of $C$, the distance 
 in $G$ between $u$ and $v$ is determined by their distance in $C$.
\end{observation}

The following fact follows from the previous observation.

\begin{observation}\label{obs:HomDistanceDetmByCycle}
 Suppose $G$ and $H$ are two graphs of odd-girth $2k+1$ and that $\phi$ is a 
homomorphism of $G$ to $H$.
 If $u$ and $v$ are two vertices of $G$ on a common $(2k+1)$-cycle of $G$, 
 then $d_H(\phi(u), \phi(v))=d_G(u,v)$.
\end{observation}

\subsection{$K_4$-minor-free graphs}

The class of $K_4$-minor-free graphs has a classic characterization as the set of \emph{partial $2$-trees}. A \emph{$2$-tree} is a graph that can be built from a $2$-vertex complete graph $K_2$ in a sequence $G_0=K_2, G_1,\ldots, G_t$ where $G_i$ is obtained from $G_{i-1}$ by adding a new vertex and making it adjacent to two adjacent vertices of $G_{i-1}$ (thus forming a new triangle). A partial $2$-tree is a graph that is a subgraph of a $2$-tree. Since $2$-trees are $2$-degenerate (that is, each subgraph has a vertex of degree at most~$2$), any $2$-tree of order $n$ has exactly $2n-3$ edges.

The following fact is well-known and will be useful (for a reference, see for example Diestel's book~\cite{D10}).

\begin{theorem}[{\cite[Proposition~8.3.1]{D10}}]\label{thm:2-trees}
A graph is $K_4$-minor-free if and only if it is a partial $2$-tree.
\end{theorem}

Note that the set of edge-maximal $K_4$-minor-free graphs coincides with the set of $2$-trees.

Another alternative definition of $K_4$-minor-free (multi)graphs is via the classic notion of 
\emph{series-parallel graphs}, indeed a graph is $K_4$-minor-free if and only if each biconnected 
component is a series-parallel graph~\cite{D10}. A (multi)graph $G$ is series-parallel if it 
contains two vertices $s$ and $t$ such that $G$ can be built using the following inductive 
definition: (i) an edge whose endpoints are labelled $s$ and $t$ is series-parallel; (ii) the graph obtained from two series-parallel graphs by identifying their $s$-vertices and their $t$-vertices, and labeling the new vertices $s$ and $t$ correspondingly is series-parallel (\emph{parallel operation}); 
(iii) the graph obtained from two series-parallel graphs by identifying vertex $s$ from one of them with vertex $t$ from the other and removing their labels is series-parallel (\emph{series operation}). Thus, the abbreviation $\SPG$ is commonly used to denote the class of $K_4$-minor-free graphs. We will also use this notation, as well as $\SP{2k+1}$ to denote the class of $K_4$-minor-free graphs of odd-girth at least $2k+1$.

We note that many homomorphism problems are studied on the class of $K_4$-minor-free graphs in the literature. A notable article is~\cite{NN07}, in which Ne\v{s}et\v{r}il and Nigussie prove that the homomorphism order restricted to the class $\SPG$ is \emph{universal}, that is, it contains an isomorphic copy of each countable order as an induced sub-order. In other words, the concept of homomorphisms inside the class of $K_4$-minor-free graphs is far from being trivial, although ordinary vertex-colouring on this class is a simple problem. As we will see next, circular colouring is also a nontrivial problem on $K_4$-minor-free graphs.

\subsection{Circular chromatic number}

Given two integers $p$ and $q$ with $gcd(p,q)=1$, the \emph{circular clique} $C_{p,q}$ is the graph on vertex set $\{0,\ldots, p-1\}$ with $i$ adjacent to $j$ if and only if $q \leq i-j \leq p-q$. A homomorphism of a graph $G$ to $C_{p,q}$ is called a $(p,q)$-colouring, and the \emph{circular chromatic number} of $G$, denoted $\chi_c(G)$, is the smallest rational $p/q$ such that $G$ has a $(p,q)$-colouring. Since $C_{p,1}$ is the complete graph $K_{p}$, we have $\chi_c(G)\leq \chi(G)$.

The (supremum of) circular chromatic number of $K_4$-minor-free graphs of given odd-girth was completely determined in the series of papers~\cite{HZ00,PZ02a,PZ02b}. For the case of triangle-free $K_4$-minor-free graphs, it is proved in~\cite{HZ00} that $\SP{5}$ is bounded by the circular clique $C_{8,3}$, also known as the Wagner graph. Furthermore, each graph in $\SP{7}$ has circular chromatic number at most $5/2$~\cite{PZ02a} (equivalently $\SP{7}$ is bounded by the $5$-cycle). The latter result is shown to be optimal in the following theorem, that we will use in one of our proofs.

\begin{theorem}[Pan and Zhu \cite{PZ02b}]\label{thm:circ_chr_nr}
For any $\epsilon>0$, there is a graph of $\SP{7}$ with circular chromatic number at least $5/2-\epsilon$.
\end{theorem}

Given a graph $G$, a rational $p/q$, a positive integer $k$ and a homomorphism $h$ 
of $G$ to $C_{p,q}$, a \emph{$pk$-tight cycle} in $G$ with respect to $h$ is a 
cycle $C$ of length $pk$ such that any two consecutive vertices $u_i$ and 
$u_{i+1}$ of $C$ satisfy $h(u_{i+1})-h(u_i)=q \bmod p$. The following 
proposition is important when studying the circular chromatic number of a graph, 
and will be useful to us.

\begin{proposition}[Guichard \cite{G93}]\label{prop:tight-cycle}
Let $G$ be a graph with $\chi_c(G)=p/q$. Then, in any homomorphism $h$ of $G$ to $C_{p,q}$, there is positive integer $k$ such that $G$ contains a tight $pk$-cycle with respect to $h$.
\end{proposition}

\section{Necessary and sufficient conditions for bounding $\SP{2k+1}$}\label{sec:maintheorem}

In this section, we develop necessary and sufficient conditions under
which a graph $B$ of odd-girth $2k+1$ is a bound for $\SP{2k+1}$. We
first introduce some notions that are important to express these
conditions.

\subsection{Preliminaries}

The aforementioned conditions are derived from a specific family of $K_4$-minor-free graphs defined below.

\begin{definition}\label{def:Tpqr}
 Given any positive integer $k$ and integers $p,q,r$ between 1 and
 $k$, $T_{2k+1}(p,q,r)$ is the graph built as follows. Let $u,v,w$ be
 three vertices. Join $u$ to $v$ by two disjoint paths of length $p$
 and $2k+1-p$, $u$ to $w$ by two disjoint paths of length $q$ and
 $2k+1-q$ and $v$ to $w$ by two disjoint paths of length $r$ and
 $2k+1-r$.
\end{definition}

\begin{figure}[ht]
\begin{center}
\resizebox{5cm}{4.5cm}{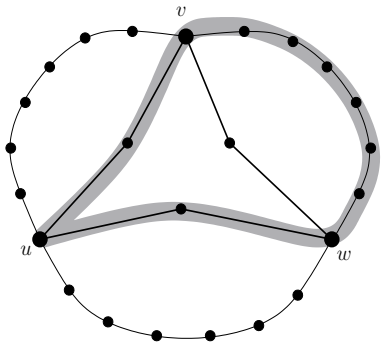 }
\end{center}
\caption{The graph $T_9(2,2,2)$; an $11$-cycle goes trough $u,v,w$.}
\label{OD222}\label{T_9(2,2,2)}
\end{figure}

The graph $T_{2k+1}(p,q,r)$ is $K_4$-minor-free (for any values of $k, p, q$ and $r$). We are mostly interested in the case where this graph has odd-girth $2k+1$, that is, when there is no odd-cycle going through $u$, $v$ and $w$ of length less than $2k+1$. A triple $\{p,q,r\}$ is called \emph{$k$-good} if the odd-girth of $T_{2k+1}(p,q,r)$ is at least $2k+1$. Clearly, the shortest cycle of $T_{2k+1}(p,q,r)$ going through all three of $u$, $v$ and $w$ has length $p+q+r$ (see Figure~\ref{T_9(2,2,2)} for an example). Note that there are exactly eight cycles going through all three of $u$, $v$ and $w$, four of even length and four of odd length. Thus, deciding whether $\{p,q,r\}$ is $k$-good is an easy task. However, we provide an easier necessary and sufficient condition, as follows.

\begin{proposition}\label{prop:good_triple}
Let $k$ be a positive integer and $p,q,r$ be three integers between
$1$ and $k$. We have the following.
\begin{enumerate}
\item[(i)] If $p+q+r$ is odd, then $\{p,q,r\}$ is $k$-good if and only if $p+q+r\geq 2k+1$.
\item[(ii)] If $p+q+r$ is even, then $\{p,q,r\}$ is $k$-good if and only if $p,q,r$ satisfy the triangular inequalities ($p \leq q + r$, $q \leq p + r$ and $r\leq p+q$).
\end{enumerate}
\end{proposition}
\begin{proof}
As mentioned before, $p+q+r$ is the length of a shortest cycle of $T_{2k+1}(p,q,r)$ containing all three of $u$, $v$ and $w$. Thus, (i) follows directly from the definition of a $k$-good triple. For (ii) we may assume, without loss of generality, that $p\leq q\leq r$. Then, a shortest odd cycle of $T_{2k+1}(p,q,r)$ containing all three of $u$, $v$ and $w$ is of length $p+q+2k+1-r$, which is at least $2k+1$ if and only if $p+q\geq r$.
\end{proof}

As a direct consequence of Proposition~\ref{prop:good_triple}, we obtain the following.

\begin{observation}\label{obs:special_good_triple}
Let $k$ be a positive integer and let $p$, $q$ be two integers between
1 and $k$. Then, $\{p,q,k\}$ is $k$-good if and only if $p+q\geq
k$. In particular, $\{p,k,k\}$ is always a $k$-good triple.
\end{observation}

The following definition will be central to our work.

\begin{definition}\label{def:all-k-good}
Let $k$ be a positive integer, $B$ be a graph, and $(\widetilde{B},d_B)$ a $k$-partial distance graph of $B$. For an edge $xy$ with $d_B(xy)=p$, we say that a $k$-good triple $\{p,q,r\}$ is \emph{realized} on $xy$ if there are two vertices $z_1$ and $z_2$ of $B$ with $d_B(x,z_1)=q$, $d_B(y,z_1)=r$, $d_B(x,z_2)=r$ and $d_B(y,z_2)=q$.
We say that $(\widetilde{B},d_B)$ has the \emph{all $k$-good triple property} if $E(\widetilde{B})\neq\emptyset$, and for each edge $xy$ and each $k$-good triple $T = \{p,q,r\}$ with $p=d_B(x,y)$, $T$ is realized on $xy$.
\end{definition}

We observe the following facts with respect to Definition~\ref{def:all-k-good}.

\begin{observation}\label{obs:all-k-good-implies-every-edge-in-cycle}
If a $k$-partial distance graph $(\widetilde{B},d_B)$ of some graph $B$ of odd-girth $2k+1$ has the all $k$-good triple property, then for every edge $xy$ of $\widetilde{B}$, there is a $(2k+1)$-cycle of $B$ containing both $x$ and $y$.
\end{observation}
\begin{proof}
Let $p$ denote the distance between $x$ and $y$ in $B$. By the definition of a $k$-partial distance graph, $p$ is less than or equal to $k$. Since
\begin{equation*}
p+\left\lfloor\frac{2k+1-p}{2}\right\rfloor+\left\lceil\frac{2k+1-p}{2}\right\rceil=2k+1,
\end{equation*}
Proposition~\ref{prop:good_triple}(i) tells us that 
\begin{equation*}
\left\{p,\left\lfloor\frac{2k+1-p}{2}\right\rfloor,\left\lceil\frac{2k+1-p}{2}\right\rceil\right\} \text{ is a }k\text{-good triple.}
\end{equation*}
Hence there is a vertex $z$ of $B$ such that $xz$ and $yz$ are both
edges in $\widetilde{B}$ and
\begin{equation*}
d_B(x,z)=\left\lfloor\frac{2k+1-p}{2}\right\rfloor \text{ and } d_B(y,z)=\left\lceil\frac{2k+1-p}{2}\right\rceil.
\end{equation*} 
The paths of these lengths connecting $z$ to $x$ and $z$ to $y$ in
$B$, together with a path from $x$ to $y$ of length $p$, form a closed
walk of length $2k+1$ in $B$. Such a walk must contain an odd cycle,
and since $B$ is of odd-girth $2k+1$, this walk is a $(2k+1)$-cycle in
which $x$ and $y$ are at distance~$p$.
\end{proof}

\begin{observation}\label{obs:all-k-good-implies-all-weights}
If a $k$-partial distance graph $(\widetilde{B},d_B)$ of some graph $B$ has the all $k$-good triple property, then for each $p$ with $1\leq p\leq k$, $(\widetilde{B},d_B)$ contains an edge of weight $p$.
\end{observation}
\begin{proof}
By definition, $(\widetilde{B},d_B)$ contains at least one edge $e$ that has weight $d\leq k$. By Observation~\ref{obs:special_good_triple}, $\{d,k,k\}$ is a $k$-good triple. Therefore, again by definition, $(\widetilde{B},d_B)$ contains an edge of weight $k$. Noting that $\{p,k,k\}$ is also a $k$-good triple completes the proof.
\end{proof}

\subsection{Main theorem}

We are now ready to prove the main theorem of this section. Roughly speaking, our claim is that for a graph $B$ of odd-girth~$2k+1$, the existence of a $k$-partial distance-graph of $B$ with the all $k$-good triple property is both necessary and sufficient for $B$ to be a minimal bound for $\SP{2k+1}$. The more precise statement is as follows.

\begin{theorem}\label{thm:TripleProperties}
Let $B$ be a graph and $(\widetilde{B},d_B)$ be a $k$-partial distance graph of $B$. If $B$ is of odd-girth $2k+1$ and $(\widetilde{B},d_B)$ has the all $k$-good triple property, then $B$ is a bound for $\SP{2k+1}$. Furthermore, if $B$ is a bound for $\SP{2k+1}$ and $B$ has odd-girth $2k+1$, then there exists a $k$-partial distance graph $(\widetilde{B},d_B)$ of $B$ which has the all $k$-good triple property.
\end{theorem}
\begin{proof}
For the first part of the claim, suppose that $B$ is of odd-girth $2k+1$ and that $(\widetilde{B},d_B)$ has the all $k$-good triple property. We show that every $K_4$-minor-free graph of odd-girth at least $2k+1$ admits a homomorphism to $B$.

Let $G$ be a graph of order $n$ in $\SP{2k+1}$. By Theorem~\ref{thm:2-trees}, $G$ is a partial $2$-tree, hence it is obtained from a $2$-tree $H$ by removing some edges. Let $H$ be a $2$-tree such that $V(H) = V(G)$ and $E(G)$ is a subset of $E(H)$. The $2$-tree structure of $H$ gives us a linear ordering of its vertices $v_0, v_1, \ldots ,v_n$ such that the edge $v_0v_1$ is in $H$ and for any $i$ between $2$ and $n$, $v_i$ has exactly two neighbours among $\{v_j : j < i\}$ in $H$ that form an edge in $H$. For $i$ between $1$ and $n$, let $H_i$ denote the graph $H$ induced by the vertices $v_0,v_1, \ldots, v_i$.

For any edge $xy$ in $H$, we define $\omega(xy)=\min \{ d_G(x,y), k\}$. We now build a weighted graph homomorphism of $(H,\omega)$ to $(\widetilde{B},d_B)$. For this, we build a weighted graph homomorphism of $(H_i,\omega)$ to $(\widetilde{B},d_B)$ for each $i$ from $1$ to $n$. When $i$ is stricly less than $n$, $\omega$ is understood as its restriction to $V(H_i)$.
 
Since $(\widetilde{B},d_B)$ has the all $k$-good triple property, by Observation~\ref{obs:all-k-good-implies-all-weights}, there exists an edge $xy$ of weight $\omega(v_0v_1)$.  Let us map $v_0$ and $v_1$ to $x$ and $y$ respectively. Then $(H_1,\omega)$ admits a weighted graph homomorphism to $(\widetilde{B},d_B)$.

Suppose that for some $i$ between $1$ and $n-1$, $(H_i,\omega)$ has a weighted graph homomorphism $\phi$ to $(\widetilde{B},d_B)$. We shall extend it by selecting an adequate image for $v_{i+1}$. Let $x$ and $y$ be the two neighbours of $v_{i+1}$ in $H_{i+1}$. They form an edge in $H_i$ and thus  $d_B(\phi(x),\phi(y))=\omega(xy)$. Let us define $p=\omega(xy)$, $q=\omega(v_{i+1}x)$ and $r=\omega(v_{i+1}y)$. 

We claim that $\{p,q,r\}$ is a $k$-good triple. To prove our claim we consider three possibilities: 

\begin{itemize}
\item At least two of $p,q$ and $r$ are equal to $k$: this case follows from Observation~\ref{obs:special_good_triple}. 

\item Exactly one of $p,q$ and $r$ is equal to $k$: let us say $p=k$ without loss of generality. It means that the distance in $G$ between $x$ and $y$ is greater than or equal to $k$ while $q$ and $r$ are the actual distances between $v_{i+1}$ and $x$ and $y$. The triangular inequality is satisfied by the distances in $G$ so that $q+r \geq d_G(x,y) \geq k = p$. By Observation~\ref{obs:special_good_triple}, $\{p,q,r\}$ is a $k$-good triple.

\item None of $p,q$ and $r$ is equal to $k$: then $p,q$ and $r$ are the actual distances and verify the triangular inequalities. Moreover, if $p+q+r$ is odd, this sum is at least $2k+1$ since the odd-girth of $G$ is at least $2k+1$. By Proposition~\ref{prop:good_triple}, $\{p,q,r\}$ is a $k$-good triple.
\end{itemize}

Then, since $(\widetilde{B},d_B)$ has the all $k$-good triple property, there exists a vertex $z$ in $V(B)$ such that $d_B(z,\phi(x))=q$ and $d_B(z,\phi(y))=r$. Let $\phi(v_{i+1})$ be this vertex $z$. Now $\phi$ is a homomorphism of $(H_{i+1},\omega)$ to $(\widetilde{B},d_B)$.

By the end of the process, we have proved the existence of (and built) a homomorphism of $(H,\omega)$ to $(\widetilde{B},d_B)$. The edges of $G$ are exactly those edges of $H$ with weight~$1$. Since $\phi$ sends these edges to edges of $(\widetilde{B},d_B)$ with weight~$1$, it means that $\phi$ induces a homomorphism of $G$ to $B$.

\bigskip

To prove the second part of the theorem, we first assume that $B$ is a minimal bound of odd-girth $2k+1$ for $\SP{2k+1}$; our aim is to build a $k$-partial distance graph of $B$ with the all $k$-good triple property.

Let $\mathcal C$ be the class of $k$-partial distance graphs $(\widetilde{G},d_G)$ satisying:\\
(i) $G\in \SP{2k+1}$;\\
(ii) $\widetilde{G}$ is a $2$-tree;\\
(iii) $(\widetilde{G},d_G)$ is a $k$-partial distance graph of $G$;\\
(iv) for every edge $uv$ of $\widetilde{G}$, $u$ and $v$ lie on a common $(2k+1)$-cycle of $G$.

\medskip

It is clear that $\mathcal C$ is nonempty, for example for $G=C_{2k+1}$ it is easy to construct a corresponding $k$-partial distance graph satisfying (i)--(iv).

Our aim is to show that if $B^*$ ($B\subseteq B^*$) is minimal such that $(B^*,d_B)$ bounds $\mathcal C$, then $(B^*,d_B)$ has the all $k$-good triple property. We first need to show that such a $B^*$ exists. To this end, we show that $(K_{|V(B)|},d_B)$ is a bound for $\mathcal C$.

Indeed, since $B$ bounds $\SP{2k+1}$, there exists a homomorphism $f:G\to B$. We claim that $f$ is also a weighted graph homomorphism of $(\widetilde{G},d_G)$ to $(K_{|V(B)|},d_B)$. Clearly, $f$ preserves edges of weight~$1$. Now, let $uv$ be an edge of weight $p\geq 2$ in $\widetilde{G}$. By Property~(iv), $u$ and $v$ lie on a common $(2k+1)$-cycle of $G$. By Observation~\ref{obs:DistanceDetmByCycle}, we have $d_G(u,v)=d_B(f(u),f(v))$, that is, $uv$ is mapped to an edge of weight $p$ which shows the claim for $f$.

Now, consider a minimal graph $B^*$ with $B\subseteq B^*$ such that $(B^*,d_B)$ bounds the class $\mathcal C$. By Property~(iv), every edge of a weighted graph in $\mathcal C$ has weight at most~$k$, therefore this is also the case for $B^*$. In other words, $(B^*,d_B)$ is a $k$-partial distance graph of $B$. We will show that $(B^*,d_B)$ has the all $k$-good triple property.

Clearly, we have $E(B^*)\neq\emptyset$. Therefore, assume by contradiction that for some edge $xy$ with $d_B(x,y)=p$ and some $k$-good triple $\{p,q,r\}$, there is no vertex $z$ in $B^*$ with $xz,yz\in E(B^*)$, $d_B(x,z)=q$ and $d_B(y,z)=r$. By minimality of $B^*$, there exists a weighted graph $(\widetilde{G_{xy}},d_{G_{xy}})$ of $\mathcal C$ such that for any homomorphism $f$ of $(\widetilde{G_{xy}},d_{G_{xy}})$ to $(B^*,d_B)$, there is an edge $ab$ of $\widetilde{G_{xy}}$ of weight $p$, $f(a)=x$ and $f(b)=y$.

We now build a new weighted graph from $(\widetilde{G_{xy}},d_{G_{xy}})$ as follows. Let $\widehat{T}$ be a $2$-tree completion of the graph $T=T_{2k+1}(p,q,r)$ where $\widehat{T}$ contains the triangle $uvw$; this triangle has weights $p$, $q$ and $r$ in $(\widehat{T},d_T)$. Then, for each edge $ab$ of $\widetilde{G_{xy}}$ with $d_{G_{xy}}(a,b)=p$, we add a distinct copy of $(\widehat{T},d_T)$ to $(\widetilde{G_{xy}},d_{G_{xy}})$ by identifying the edge $ab$ with the edge $uv$ of $(\widehat{T},d_T)$ (that both have weight $p$). It is clear that the resulting weighted graph, that we call $(\widehat{G'_{xy}},d_{G'_{xy}})$, belongs to the class $\mathcal C$. Moreover, $(\widetilde{G_{xy}},d_{G_{xy}})$ is a subgraph of $(\widehat{G'_{xy}},d_{G'_{xy}})$ (indeed the new vertices added in the construction have not altered the distances between original vertices of $G_{xy}$). Thus, there exists a homomorphism $\phi$ of $(\widehat{G'_{xy}},d_{G'_{xy}})$ to $(B^*,d_B)$, whose restriction to the subgraph $(\widetilde{G_{xy}},d_{G_{xy}})$ is also a homomorphism. Therefore, by the choice of $G_{xy}$, at least one pair $a,b$ of vertices of $\widehat{G'_{xy}}$ with $d_{G'_{xy}}(a,b)=p$ is mapped by $\phi$ to $x$ and $y$, respectively. But then, the copy of $\widehat{T_{2k+1}}(p,q,r)$ added to $G$ for this pair forces the existence of the desired triangle on edge $xy$ in $B^*$, which is a contradiction.

To complete the proof, if $B$ is not minimal, consider a minimal subgraph $B_m$ of $B$ that is a bound for $\SP{2k+1}$. As proved above, there is a partial distance graph $(B_m^*,d_{B_m})$ of $B_m$ with the all $k$-good triple property. Using Observation~\ref{obs:all-k-good-implies-every-edge-in-cycle} and Observation~\ref{obs:DistanceDetmByCycle}, we conclude that $d_{B_m}$ and $d_B$ coincide on $E(B_m^*)$. Therefore, $(B_m^*,d_{B_m})$ is also a partial distance graph of $B$, which completes the proof.
\end{proof}

\subsection{Some properties of minimal bounds}

We now show that a minimal bound must satisfy some simple structural conditions. The following lemmas are examples of such conditions that are useful in the theoretical investigation of minimal bounds, as we will see in Section~\ref{sec:smallvalues}.

\begin{lemma}\label{Degree2Vertex-cycle}
Let $k$ be strictly greater than 1. If $B$ is a minimal bound (in
terms of subgraph inclusion) of odd-girth $2k+1$ for $\SP{2k+1}$, then
any degree~$2$-vertex belongs to a $6$-cycle.
\end{lemma}
\begin{proof}Let $u$ be a vertex of degree~$2$ in $B$, with $v,w$ its two neighbours. By the second part of Theorem~\ref{thm:TripleProperties}, we can assume that the partial distance graph $(\widetilde{B},d_B)$ of $B$ with the all $k$-good triple property contains all edges of $B$, in particular, the edge $uv$. The triple $\{1,1,2\}$ is $k$-good and hence, it must be realizable on edge $uv$, which implies that the edge $vw$ of weight~$2$ belongs to $(\widetilde{B},d_B)$. But then, the only way to realize the $k$-good triple $\{2,2,2\}$ on $vw$ is if $u$ is part of a $4$-cycle or a $6$-cycle. But if $u$ is part of a $4$-cycle, then $B$ is not a core, contradicting its minimality.
\end{proof}

In the following lemma, we remark that the claim holds with respect to minimality in terms of \emph{induced} subgraph inclusion, which is stronger than the minimality condition of Theorem~\ref{thm:TripleProperties} and Lemma~\ref{Degree2Vertex-cycle} (which is just about subgraph inclusion).

\begin{lemma}\label{Adjacent2Vertices}
Let $k$ be strictly greater than 1. If $B$ is a minimal bound (in
terms of induced subgraph inclusion) of odd-girth $2k+1$ for
$\SP{2k+1}$, then the set of degree $2$-vertices of $B$ forms an
independent set.
\end{lemma}
\begin{proof}
By the hypothesis, $B$ has a spanning subgraph $B'$ that is a minimal bound in terms of subgraph inclusion. Then $B'$ is a core and, therefore, has minimum degree at least~$2$. If two adjacent vertices of $B$ have degree~$2$, then they must have degree~$2$ in $B'$. By Lemma~\ref{Degree2Vertex-cycle}, they must be part of the same $6$-cycle. But then $B'$ is not a core, a contradiction.
\end{proof}

The following lemma requires an even stronger minimality condition than the one in Lemma~\ref{Adjacent2Vertices}.

\begin{lemma}\label{new_lemma_6-cycle}
Let $k$ be striclty greater than 1. If $B$ is a minimal bound (in
terms of induced subgraph inclusion) of odd-girth $2k+1$ for
$\SP{2k+1}$ that has no homomorphism to any smaller graph of
odd-girth~$2k+1$, then any $6$-cycle $C$ of $B$ can contain at most
two degree~$2$ vertices. If furthermore $C$ contains two such
vertices, then they must be at distance~$3$ in $C$ (and in $B$).
\end{lemma}
\begin{proof}
Let $u$ be a degree~$2$ vertex of $B$ belonging to a $6$-cycle $C:uvwxyz$. By 
Lemma~\ref{Adjacent2Vertices}, $v$ and $z$ must have degree at least~$3$. Assume for a contradiction that $y$ or $w$ (say $w$) has degree~$2$. Then, $u$ and $w$ must belong to a common $(2k+1)$-cycle (otherwise, identifying $u$ and $w$ would give a homomorphism to a smaller graph of odd-girth~$2k+1$). Moreover, since $deg(u)=deg(w)=2$, this $(2k+1)$-cycle uses four edges of $C$. Thus, replacing them with the two other edges of $C$ yields a $(2k-1)$-cycle, a contradiction.
\end{proof}

\section{A polynomial-time algorithm to check whether a given graph of odd-girth $2k+1$ bounds 
$\SP{2k+1}$}\label{sec:algo}

In this section, we show that the characterization of Theorem~\ref{thm:TripleProperties} is sufficiently strong to imply the existence of a polynomial-time algorithm that checks whether a given graph of odd-girth $2k+1$ bounds $\SP{2k+1}$. We describe this algorithm as Algorithm~\ref{algo}.

\begin{algorithm}{Deciding whether a graph of odd-girth $2k+1$ bounds $\SP{2k+1}$.}
\begin{algorithmic}[1]
  \Require{An integer $k$, a graph $B$.}
  \Statex
  \State Compute the odd-girth $g$ of $B$.\label{alg:OG}
  \If{$g\neq 2k+1$}
  \State\Return NO\Comment{\emph{($B$ is not a bound)}}\label{alg:OG-bis}
  \EndIf
  \State Compute the distance function $d_B$ of $B$.\label{alg:distances}
  \State Let $(\widetilde{B},d_B)$ be the $k$-partial distance graph of $B$ obtained from the 
complete distance graph of $B$ by removing all edges of weight more than $k$.
  \State Compute the set $\mathcal T_k$ of $k$-good triples.
  \For{$e=xy$ in $E(\widetilde{B})$ with $d_B(xy)\leq k$}\label{alg:loop}
  \Let{$p$}{$d_B(xy)$}
  \For{each $k$-good triple $t=\{p,q,r\}\in\mathcal T_k$ containing $p$}
  \If{there is no pair $z,z'$ of $V(B)$ with $d_B(xz)=d_B(yz')=q$ and 
$d_B(yz)=d_B(xz')=r$}\label{alg:check_edge}
  \Statex\Comment{\emph{(the edge $e$ fails for the all $k$-good triple property)}}
  \If{$d_B(xy)\geq 2$}
  \Let{$\widetilde{B}$}{$(\widetilde{B}-uv)$}\label{alg:update}
  \State Restart the loop (Step~\ref{alg:loop}).
  \Else\Comment{\emph{($uv$ is an edge of $B$)}}
  \State\Return Algorithm~\ref{algo}($k$,$B-uv$).\Comment{\emph{(Recursive call with a smaller 
graph.)}} \label{alg:restart}
  \EndIf  
  \EndIf
  \EndFor
  \EndFor
  \State\Return YES\Comment{\emph{($\widetilde{B}$ is a certificate)}}
\end{algorithmic}
\label{algo}  
\end{algorithm}

We now analyze Algorithm~\ref{algo}.

\begin{theorem}
Algorithm~\ref{algo} checks in time $O(mn^5k^3)=O(mn^8)$ whether a given graph $B$ of odd-girth $2k+1$ with $n$ vertices and $m$ edges bounds $\SP{2k+1}$.
\end{theorem}
\begin{proof}
We prove that Algorithm~\ref{algo} is correct, that is, it returns ``YES'' if and only if $B$ is a bound for $\SP{2k+1}$. 

Assume first that Algorithm~\ref{algo} returns ``YES''. Then, it has found a $k$-partial distance graph of some subgraph $B'$ of $B$ and a $k$-partial distance graph $(\widetilde{B'},d_{B'})$ of $B$ such that each edge of $\widetilde{B'}$ has passed the check of Step~\ref{alg:check_edge}. Therefore, $(\widetilde{B'},d_{B'})$ has the all $k$-good triple property and by Theorem~\ref{thm:TripleProperties}, $B'$ is a bound --- and so is $B$.

Assume now that $B$ is a bound for $\SP{2k+1}$. Then, it contains a minimal subgraph $B'$ that is also a bound. By Theorem~\ref{thm:TripleProperties}, there is a $k$-partial distance graph $(\widetilde{B'},d_{B'})$ of $B'$ having the all $k$-good triple property.

We will now show that no edge $xy$ of $\widetilde{B'}$ will ever be deleted by the algorithm (that is, $xy$ will always succeed the check of Step~\ref{alg:check_edge}). By contradiction, assume that $xy$ is the first edge of $\widetilde{B'}$ that does not succeed the check of Step~\ref{alg:check_edge}. By Observation~\ref{obs:all-k-good-implies-every-edge-in-cycle}, for each edge $uv$ of $\widetilde{B'}$, there is a $(2k+1)$-cycle $C$ of $B'$ (and hence $B$) containing both $x$ and $y$ (recall that by definition $d_{B'}(uv)\leq k$). By Observation~\ref{obs:DistanceDetmByCycle}, $d_{B'}(xy)=d_B(xy)\leq k$. By our assumption on $xy$ being the first edge to be deleted, all edges of $\widetilde{B'}$ are still present in $\widetilde{B}$ at this step of the execution. Furthermore, the distances in the copy of $B'$ in $B$ are the same as the distances in $B'$ and hence the edges of $\widetilde{B'}$ in $\widetilde{B}$ have the same weights as in $\widetilde{B'}$. Hence, $xy$ cannot fail the check of Step~\ref{alg:check_edge}, which is a contradiction. 

By the previous paragraph, Algorithm~\ref{algo} will never delete any edge of $\widetilde{B'}$ from $\widetilde{B}$. Therefore, even in a possible recursive call at Step~\ref{alg:restart}, the input graph will always have odd-girth $2k+1$ and the $k$-partial distance graph will never become empty. Thus, Algorithm~\ref{algo} will never return ``NO''. Therefore, at some step, there will be a $k$-partial distance graph (of some subgraph of $B$) having the all $k$-good triple property, and Algorithm~\ref{algo} will return ``YES''. Thus, Algorithm~\ref{algo} is correct.

Now, for the running time, note that there are $O(m)$ recursive calls to the algorithm (Step~\ref{alg:restart}) since each call corresponds to the deletion of one edge of $B$. For each call, we have the computation of the odd-girth of $B$ at Step~\ref{alg:OG}, which can be done in time $O(n(m+n\log n))=O(n^3)$ (see for example the literature review in~\cite{LP05}). Algorithm~\ref{algo} computes all distances at Step~\ref{alg:distances} and then creates $\widetilde{B}$, which can be done in $O(n^3)$ steps using Dijkstra's algorithm. Then, the loop of Step~\ref{alg:loop} is over $O(n^2)$ pairs, and for each pair, we have $O(k^3)$ $k$-good triples to check; each check needs to go through $O(n)$ vertices of $B$. Hence, one iteration of the loop takes $O(nk^3)$ steps. However, the loop may be restarted $O(n^2)$ times at Step~\ref{alg:update} (at most once for each of the $O(n^2)$ edges of $\widetilde{B}$), hence there may be $O(n^4)$ total iterations of the loop. Therefore, each recursive call to the algorithm may take $O(n^5k^3)$ time, and the total running time is $O(mn^5k^3)$. This is also $O(mn^8)$, indeed $k=O(n)$ because $B$ must contain a $(2k+1)$-cycle (to be more precise, if this is not the case, the algorithm stops at Step~\ref{alg:OG-bis}).
\end{proof}

Given an input graph $B$ of odd-girth $2k+1$, when our algorithm returns ``NO'', it would be interesting to produce, as an explicit NO-certificate, a $K_4$-minor-free graph of odd-girth $2k+1$ which does not admit a homomorphism to $B$. Algorithm~\ref{algo} could provide such a certificate, the rough ideas are as follows. Since $B$ is a NO-instance, Algorithm~\ref{algo}, after deleting a sequence $e_1, e_2, \ldots, e_m$ of weighted edges, has considered a graph $B_m$ which is either bipartite or has odd-girth larger than $2k+1$. Thus, $C_{2k+1}$ is a member of $\SP{2k+1}$ which does not admit a homomorphism to $B_m$. Let $\widehat{C_{2k+1}}$ be a $2$-tree containing $C_{2k+1}$ as a spanning subgraph and let $G_m=(\widehat{C_{2k+1}}, d_{C_{2k+1}})$ be the corresponding $k$-partial distance graph of $C_{2k+1}$. Starting from this, and in reverse order from $i=m-1$ to $i=1$, we build a weighted graph $G_i$ from $G_{i+1}$ as follows. Assume that in Algorithm~\ref{algo}, the edge $e_i$ of weight $p$ has been deleted because of no realized $k$-good triple $\{p,q,r\}$ on $e_i$. Let $(\widehat{T},d_T)$ be a weighted $2$-tree completion of the graph $T_{2k+1}(p,q,r)$ where $\widehat{T}$ includes the edge $uv$ of weight $p$ of $T_{2k+1}(p,q,r)$. Then, on each edge $xy$ of weight $p$ of the weighted graph $G_{i+1}$, glue two disjoint copies of $(\widehat{T},d_T)$ (one copy identifing $u$ with $x$ and $v$ with $y$, and one copy identifying $u$ with $y$ and $v$ with $x$). At the final step, we obtain a weighted graph $G_1$ whose subgraph induced by the edges of weight~$1$ does not admit a homomorphism to $B$. Note however that the order of this graph could be super-polynomial in $|V(B)|$; we do not know if one can create such a NO-certificate of order polynomial in $|V(B)|$.

We conclude this section by asking for a similar (not necessarily polynomial-time) algorithm as Algorithm~\ref{algo} for the case of planar graphs. Let $\mathcal P_{2k+1}$ be the class of planar graphs of odd-girth at least $2k+1$.

\begin{problem}\label{pb:planar}
For a fixed $k$, give an explicit algorithm which decides whether an input graph of odd-girth $2k+1$ bounds $\mathcal P_{2k+1}$. 
\end{problem}

For $k=1$ and by the virtue of the Four-Colour Theorem, one only needs to check whether $B$ contains $K_4$ as a subgraph. We observe that the existence of an algorithm that does not use the Four-Colour Theorem is related to Problem~2.1 of the classic book on graph colouring by Jensen and Toft~\cite{JT95}.

For other values of $k$, we do not know of any explicit algorithm. Note however, that a hypothetic algorithm exists: for a given graph $B$ of odd-girth~$2k+1$ not bounding $\mathcal P_{2k+1}$, let $f_k(B)$ be the smallest order of a graph in $\mathcal P_{2k+1}$ with no homomorphism to $B$, and let $f_k(n)$ be the maximum of $f_k$ over all such graphs of order~$n$. Then, given $B$ with odd-girth~$2k+1$, one may simply check, for all graphs in $\mathcal P_{2k+1}$ of order at most~$f_k(|V(B)|)$, whether it maps to $B$. Since $f_k$ is well-defined, this is a finite-time algorithm, but it relies on the knowledge of $f_k$. Note that part of Problem~2.1 of~\cite{JT95} consists in giving an upper bound on $f_1(K_4)$ without using the Four-Colour Theorem, which is already a difficult problem.

\section{General families of bounds}\label{sec:general}

In this section, we exhibit three bounds of odd-girth $2k+1$ for the class $\SP{2k+1}$.

\subsection{Projective hypercubes}

Projective hypercubes are well-known examples of very symmetric graphs. Property~$(ii)$ of the following lemma is to claim that projective hypercubes are \emph{distance-transitive}, which is a well-known fact. This is needed for a proof of Property~$(iii)$, hence we include a proof for the sake of completeness. 

\begin{lemma}\label{propertiesofPC}
The projective hypercube $PC(2k)$ has the following properties.
\begin{itemize}
\item[(i)] For each pair $x,y$ of vertices of $PC(2k)$, $d(x,y)\leq k$; furthermore, $x$ and $y$ belong to at least one $(2k+1)$-cycle.
 
\item[(ii)] If $d(u,v)=d(x,y)$ for some vertices $u,v,x,y$ of $PC(2k)$, then there is an automorphism of $PC(2k)$ which maps $u$ to $x$ and $v$ to $y$ (in other words, $PC(2k)$ is distance-transitive).

 \item[(iii)] Let $\{p,q,r\}$ be a $k$-good triple with $1\leq p,q,r \leq k$. Suppose $x$ and $y$ are two vertices of $PC(2k)$ at distance $p$. Let $\phi(u)=x$ and $\phi(v)=y$ be a mapping of two (main) vertices of $T_{2k+1}(p,q,r)$. Then, $\phi$ can be extended to a homomorphism of $T_{2k+1}(p,q,r)$ to $PC(2k)$.
\end{itemize}

\end{lemma}

\begin{proof}
We use the Cayley representation of the projective hypercube.

\medskip

\noindent (i). Let $D$ be the set of coordinates at which $x$ and $y$ differ, and $\overline{D}$ be the complement of $D$, that is, the set of coordinates at which $x$ and $y$ do not differ. Then $x=y+\displaystyle \sum_{i\in D} e_i=y+J+\sum_{i \in {\overline{D}}}e_i$. Let $P_1$ be the path connecting $x$ and $y$ by adding elements of $D$ to $x$ in a consecutive way. Similarly, let $P_2$ be the path connecting $x$ and $y$ by adding elements of ${\overline D}\cup \{J\}$. The smaller of $P_1$ and $P_2$ provides the distance between $x$ and $y$ and the union of the two is an example of a $(2k+1)$-cycle containing both $x$ and $y$.

\medskip

\noindent(ii). We need a more symmetric representation of $PC(2k)$. Let $S'$ be the set of $2k+1$ elements of $\mathbb{Z}_2^{2k+1}$, each with exactly two $1$'s which are consecutive in the cyclic order. Then, the Cayley graph $(\mathbb{Z}_2^{2k+1},S')$ has two connected components, each isomorphic to $PC(2k)$. It is now easy to observe that any permutation of $S'$ (equivalently, any permutation of $S$ in the original form) induces an automorphism of $PC(2k)$. Thus, to map $uv$ to $xy$, we could first map $u$ to $x$ by the automorphism $\phi(t)=t+x-u$. Then, composing $\phi$ with a permutation of $S'$ which maps $\phi(v)$ to $y$ will induce an automorphism that maps $u$ to $x$ and $v$ to $y$.

\medskip

\noindent (iii). To prove the third claim, we first give a homomorphism of $T_{2k+1}(p,q,r)$ to $PC(2k)$ (not necessarily satisfying $u\mapsto x$ and $v\mapsto y$). Since both these graphs are of odd-girth $2k+1$, by Observation~\ref{obs:HomDistanceDetmByCycle}, such a homomorphism must preserve the distances between the three vertices $u$, $v$ and $w$ of $T_{2k+1}(p,q,r)$. The proof will then be completed using the distance-transitivity of $PC(2k)$ (proved in Part~$(ii)$).

Each choice of $a\in \{p, 2k+1-p\}, b\in \{q, 2k+1-q\}$, $c\in \{r, 2k+1-r\}$ corresponds to a cycle of length $a+b+c$ of $T_{2k+1}(p,q,r)$ (which goes through all of $u$, $v$ and $w$), exactly four of which are odd cycles. We consider $a$, $b$ and $c$ such that the corresponding cycle is the shortest among these four odd cycles. Without loss of generality, we may assume that $a \leq b \leq c$.

Since $a+b+c$ is odd and $2k+1\leq a+b+c\leq 3k+1$, there is an integer $t$ such that
$a+b+c=2k+1+2t$ (thus $ 0 \leq t \leq k-1$). We first note that $b+c\leq 2k+1$, as otherwise $a+(2k+1-b)+(2k+1-c)< a+b+c$, contradicting the choice of $a,b,c$. Thus we have $2t\leq a$.

Let $S_1,S_2,S_3$ be a partition of $S=\{e_1,\ldots, e_{2k},J\}$ with $|S_1|=a$, $|S_2|=b-t$ and $|S_3|=c-t$. Let $S_1'$ be a subset of size $t$ of $S_1$. Let $S_2^+=S_2\cup S_1'$ and $S_3^+=S_3\cup S_1'$. Let $\nu_1= \vec 0$ (the $0$ vector in $\mathbb Z_2^{2k}$), $\displaystyle \nu_2=\sum_{\nu\in S_1} \nu$, and $\displaystyle \nu_3=\sum_{\nu\in S_2^+} \nu=\nu_2+\sum_{\nu\in S_3^+}\nu$. It is now easy to check that $d(\nu_1,\nu_2)\in \{a, 2k+1-a\}$, $d(\nu_1,\nu_3) \in \{b, 2k+1-b\}$ and $d(\nu_2,\nu_3)\in \{c, 2k+1-c\}$. Furthermore by Part~$(i)$, each pair of vertices of $PC(2k)$ is in a $(2k+1)$-cycle. Thus, the mapping $u\mapsto \nu_1$, $v\mapsto \nu_2$ and $w\mapsto \nu_3$ can be extended to a homomorphism of $T_{2k+1}(p,q,r)$ to $PC(2k)$.
\end{proof}

\begin{theorem}\label{BoundedByPC(2k)}
The projective hypercube $PC(2k)$ bounds $\SP{2k+1}$.
\end{theorem}
\begin{proof}
We show that $(K_{2^{2k}}, d_{PC(2k)})$ has the all $k$-good triple property, which by 
Theorem~\ref{thm:TripleProperties} will prove our claim. By Lemma~\ref{propertiesofPC}(i), 
$(K_{2^{2k}}, d_{PC(2k)})$ is a $k$-partial distance graph of $PC(2k)$. By 
Lemma~\ref{propertiesofPC}(ii), $PC(2k)$ is distance-transitive. It remains to prove that 
for each edge $xy$ of weight $p$ ($1\leq p\leq k$) of $(K_{2^{2k}}, d_{PC(2k)})$, every $k$-good 
triple $\{p,q,r\}$ is realized on the edge $xy$. Consider the graph $T=T_{2k+1}(p,q,r)$. We define a 
mapping $\phi$ of $T$ to $PC(2k)$ by first mapping the two vertices $u$ and $v$ of $T$ with 
degree~$4$ and at distance~$p$ in $T$ to $x$ and $y$. By Lemma~\ref{propertiesofPC}(iii), we can 
extend $\phi$ to the whole of $T$. By Observation~\ref{obs:DistanceDetmByCycle}, and by considering 
the three $(2k+1)$-cycles of $T$, we have 
$d_{PC(2k)}(x,z)=q$ and $d_{PC(2k)}(y,z)=r$. This completes the proof.
\end{proof}

Theorem~\ref{BoundedByPC(2k)} has applications to edge-colourings, that will be discussed in 
Section~\ref{sec:applis}.

\subsection{Kneser graphs}

It was recently shown~\cite{NSS15} that if $PC(2k)$ bounds the class of planar graphs of odd-girth $2k+1$ (that is, if Conjecture~\ref{PlanarToPC} holds), then it is an optimal bound of odd-girth $2k+1$ (in terms of the order). However, for $K_4$-minor-free graphs, $PC(2k)$ is far from being optimal.

Consider the hypercube $H(2k+1)$; its vertices can be labeled by subsets of a $(2k+1)$-set (call it $U$) where $X$ and $Y$ are adjacent if $X\subset Y$ and $|X|+1=|Y|$. In this notation, antipodal pairs are complementary sets. Recall that $PC(2k)$ is obtained from $H(2k+1)$ by identifying antipodal pairs of vertices. Thus, the vertices of $PC(2k)$ can be labeled by pairs of complementary subsets of $U$, or simply by a subset of size at most $k$ (the smaller of the two).

It is not difficult to observe the following (where $S\bigtriangleup T$ denotes the symmetric difference of sets $S$ and $T$).

\begin{observation}\label{obs:automorphism}
For any set $A\subseteq U$, the mapping $f_A$ defined by $f_A(\{X,\overline{X}\})=\{A\bigtriangleup X,A\bigtriangleup \overline{X}\}$ is an automorphism of $PC(2k)$.
\end{observation}

Using the above presentation of $PC(2k)$, the set of vertices at distance $k$ from $\emptyset$ are the $k$-subsets of $U$, and two such vertices are adjacent if they have no intersection (because then one is a subset of the complement of the other and the difference of sizes is~$1$). Thus, the Kneser graph $K(2k+1, k)$ (also known as \emph{odd graph}) is an induced subgraph of 
$PC(2k)$. These graphs are well-known examples of distance-transitive graphs (indeed the distance between two vertices is determined by the size of their intersection). In particular, $K(5,2)$ is the Petersen graph. We refer to \cite{GR01} for more details on this family of graphs.

The following theorem is then a strengthening of Theorem~\ref{BoundedByPC(2k)}.

\begin{theorem}\label{BoundedByK(2k+1,k)}
The Kneser graph $K(2k+1,k)$ bounds $\SP{2k+1}$.
\end{theorem}
\begin{proof}
We use Theorem~\ref{thm:TripleProperties} by showing that $(K_n,d_{K(2k+1,k)})$ with $n=|V(K(2k+1,k))|={\binom{2k+1}{k}}$ is a $k$-partial distance graph of $K(2k+1,k)$ and has the all $k$-good triple property. Since $K(2k+1,k)$ is distance-transitive, it suffices to prove that for any $k$-good triple $\{p,q,r\}$, there are three vertices in $K(2k+1,k)$ whose pairwise distances are $p$, $q$ and $r$. We will use Theorem~\ref{BoundedByPC(2k)} and the above-mentioned presentation of $K(2k+1,k)$ as an induced subgraph of $PC(2k)$. In this presentation, each vertex of $PC(2k)$ is a pair $(S,\overline{S})$ of subsets of a $(2k+1)$-set $U$, where $|S|\leq k$. We may use any of $S$ or $\overline{S}$ to denote the vertex $(S,\overline{S})$.

Let $\{p,q,r\}$ be a $k$-good triple. By Theorem~\ref{BoundedByPC(2k)}, there are three vertices $A$, $B$ and $C$ of $PC(2k)$ with $|A|\leq k$, $|B|\leq k$ and $|C|\leq k$ such that their pairwise distances are~$p$, $q$ and $r$. Our goal is to select a set $X$ of elements of $U$, such that the sizes of their symmetric differences $A\bigtriangleup X$, $B\bigtriangleup X$ and $C\bigtriangleup X$ are $k$ or $k+1$. Then, all these three new vertices belong to an induced subgraph of $PC(2k)$ isomorphic to $K(2k+1,k)$. Since by Observation~\ref{obs:automorphism} the operation is an automorphism of $PC(2k)$, we have proved the claim for the triple $\{p,q,r\}$. In fact, we will construct $X$ in four steps.

First, let $X_1$ be the set of elements of $U$ that belong to at least two of the sets $A$, $B$, and $C$, and let $A_1=A\bigtriangleup X_1$, $B_1=B\bigtriangleup X_1$ and $C_1=C\bigtriangleup X_1$. Then, $A_1$, $B_1$ and $C_1$ are pairwise disjoint.

From $A_1$, $B_1$ and $C_1$, we build three sets $A_2$, $B_2$ and $C_2$ such that the difference of the sizes of any two of them is at most~$1$. To do this, assuming $|A_1| \geq |B_1| \geq |C_1|$, we first consider a set $X_2$ of $\lfloor(|B_1|-|C_1|)/2\rfloor$ elements of $B_1$ (none of them belongs to $C_1$ since $B_1$ and $C_1$ are disjoint), and consider again the symmetric differences of the three sets with $X_2$: $A_2=A_1\bigtriangleup X_2$, $B_2=B_1\bigtriangleup X_2$ and $C_2=C_1\bigtriangleup X_2$. Now, we have $|B_2|\geq |C_2|=\lfloor(|B_1|+|C_1|)/2\rfloor$ and $|B_2|-|C_2|\leq 1$, moreover $A_2\cap B_2=\emptyset$. We repeat the operation with $A_2$ and $C_2$: let $X_3$ be a set of $\lfloor(|A_2|-|C_2|)/2\rfloor$ elements of $A_2\setminus (B_2\cup C_2)$ (such a set exists because $\lfloor(|A_2|-|C_2|)/2\rfloor=\lfloor(|A_1|-|C_1|)/2\rfloor\leq |A_1|=|A_2\setminus (B_2\cup C_2)|$). We let $A_3=A_2\bigtriangleup X_3$, $B_3=B_2\bigtriangleup X_3$ and $C_3=C_2\bigtriangleup X_3$. Now, the size of each of $A_3$, $B_3$ and $C_3$ is $s$ or $s+1$, with $s=\lfloor(|A_1|+|B_1|+|C_1|)/3\rfloor$.

Moreover, we have $A_3\cup B_3=A_1\cup B_1$. Since $A_1$ and $B_1$ are disjoint and $|A_1|\leq k$ and $|B_1|\leq k$, we have $s\leq k$. If $s=k$, we are done. Otherwise, consider the set $X_4$ of elements which are in none of $A_3$, $B_3$ and $C_3$. If $|X_4|\geq k-s$, selecting a subset $X_4'$ of $X_4$ of size~$k-s$ and adding $X_4'$ to all three of $A_3$, $B_3$ and $C_3$, we are done. Otherwise, we have $|X_4|<k-s$. Nevertheless, let $A_4=A_3\bigtriangleup X_4=A_3\cup X_4$, $B_4=B_3\bigtriangleup X_4=B_3\cup X_4$ and $C_4=C_3\bigtriangleup X_4=C_3\cup X_4$. Note that now we have $A_4\cup B_4\cup C_4=U$, and the size of each of $A_4$, $B_4$ and $C_4$ is $t$ or $t+1$, with $t<k$.

Let $n_a$ be the number of elements that are in $A_4$ but not in $B_4$ nor $C_4$, that is, $n_a=|A_4\setminus (B_4\cup C_4)|$; $n_b$ and $n_c$ are defined in the same way. Similarly, let $n_{ab}=|(A_4\cup B_4)\setminus C_4|$, and we define $n_{bc}$ and $n_{ac}$ in the same way. We claim that $n_a\geq k-t$. Note that $n_a+n_b+n_{ab}=|U|-n_c$. Hence, since $n_c\leq q+1\leq k$, we have $n_a+n_b+n_{ab}\geq k+1$. But since $n_b+n_{ab}\leq |B_4|\leq t+1$, we have $n_a\geq k-t$. The same argument shows that $n_b\geq k-t$ and $n_c\geq k-t$.

Therefore, we can select three sets of size $k-t$ in each of $A_4\setminus (B_4\cup C_4)$, $B_4\setminus (A_4\cup C_4)$ and $C_4\setminus (A_4\cup B_4)$; let $X_5$ be the union of these three sets. Then, the sets $A_5=A_4\bigtriangleup X_5$, $B_5=B_4\bigtriangleup X_5$ and $C_5=C_4\bigtriangleup X_5$ all have size $k$ or $k+1$, which completes the proof.
\end{proof}

By the definition of fractional chromatic number, the following is an immediate corollary of Theorem~\ref{BoundedByK(2k+1,k)}.

\begin{corollary}\label{cor:fractional}
 For every graph $G$ in $\SP{2k+1}$, we have $\chi_f(G)\leq 2+\frac{1}{k}$.
\end{corollary}

This bound is tight as $\chi_f(C_{2k+1})=2+\frac{1}{k}$.

During the writing of this paper, it has come to our attention that
Theorem~\ref{BoundedByK(2k+1,k)} and Corollary~\ref{cor:fractional}
were obtained, independently, in a recent preprint of Feder and
Subi~\cite{FSman} and also by Goddard and Xu~\cite{GX15}.

\subsection{Augmented toroidal grids}

In this subsection, we provide a bound of odd-girth $2k+1$ and of order $4k^2$
for $\SP{2k+1}$.

For any pair of integers $(a,b)$, let $T(a,b)$ denote the cartesian
product $C_{a}\Box C_{b}$. This graph can be seen as the {\em toroidal
 grid of dimension $a \times b$}. Figure~\ref{TG(24,24)} depicts a
representation of $T(24,24)$.

\begin{figure}[ht!]
\centering
\includegraphics[width=0.3\textwidth]{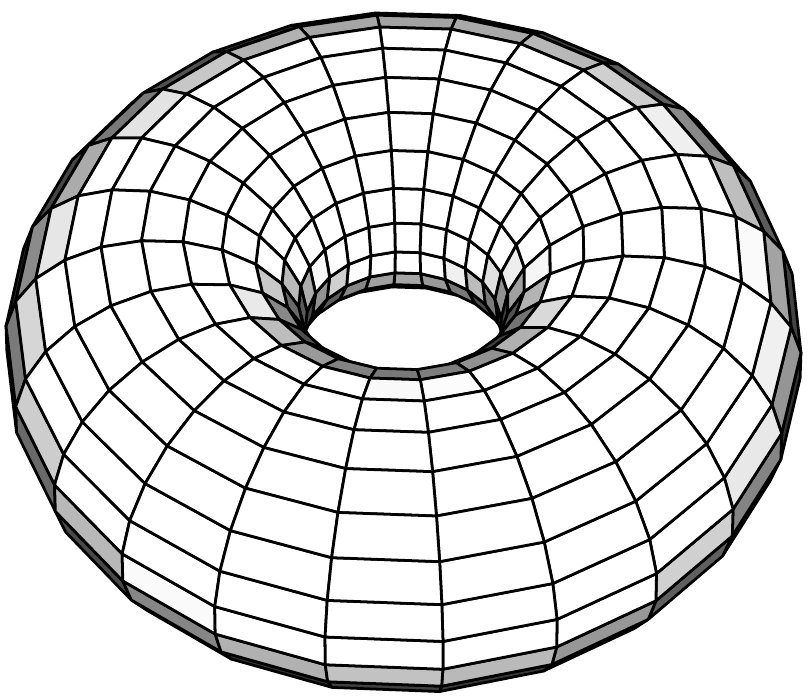}
\caption{A representation of the $24\times 24$ toroidal grid.}\label{TG(24,24)}
\end{figure}

The graph $T(2a,2b)$ is of diameter $a+b$ and, furthermore, given a vertex $v$, there is a unique vertex at distance $a+b$ of $v$ which is therefore called \emph{antipodal} of $v$; we denote it by $\overline v$.

The {\em augmented toroidal grid} of dimensions $2a$ and $2b$, denoted $AT(2a,2b)$ 
is the graph obtained from $T(2a,2b)$ by adding an edge between $v$ and $\overline v$ for each vertex $v$. We will restrict ourselves to augmented toroidal grids of equal dimensions. More formally, for any positive integer $k$, let $AT(2k,2k)$ be the graph defined on the vertex set
$\{0,1,\ldots,2k-1\}^2$ such that a pair $\{(i_1,j_1),(i_2,j_2)\}$ is an edge if

\smallskip

$i_1 = i_2$ and $\lvert j_1 - j_2 \rvert \in \{1,2k-1\}$ (vertical
edges),

\smallskip

or $j_1 = j_2$ and $\lvert i_1 - i_2 \rvert \in \{1,2k-1\}$
(horizontal edges),

\smallskip

or $i_2 - i_1 + k \in \{0,2k\}$ and $j_2 - j_1 + k \in \{0,2k\}$
(antipodal edges).\\[1mm] 
Figure~\ref{fig:toroidal} gives a representation of the graph $AT(6,6)$.

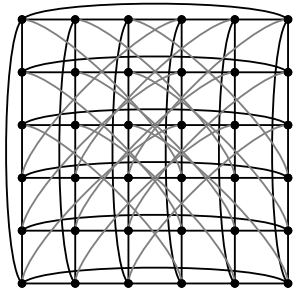
\begin{figure}[ht]
\begin{center}
\scalebox{0.7}{\begin{tikzpicture}[join=bevel,inner sep=0.5mm]

\foreach \I in {0,...,5}\foreach \J in {0,...,5}
         {\node[draw,shape=circle,fill] (\I,\J) at (1*\I,1*\J) {};}
\foreach \I in {0,...,5}\foreach \J in {0,...,4}
         {\draw[-,line width=1pt] (1*\I,1*\J+1)--(1*\I,1*\J);}
\foreach \I in {0,...,4}\foreach \J in {0,...,5}
         {\draw[-,line width=1pt] (1*\I,1*\J)--(1*\I+1,1*\J);}
\foreach \I in {0,...,5}
         {\draw[-,line width=1pt] (1*\I,0) .. controls +(-0.4,0.4) and 
+(-0.4,-0.4) .. (1*\I,5);}
\foreach \J in {0,...,5}
         {\draw[-,line width=1pt] (0,1*\J) .. controls +(0.4,0.4) and 
+(-0.4,0.4) .. (5,1*\J);}

\foreach \I in {0,1,2}\foreach \J in {0,1,2}
         {\draw[-,line width=1pt,color=gray] (1*\I,1*\J) .. controls +(0,0.7) 
and +(-0.7,0) .. (1*\I+3,1*\J+3);}
\foreach \I in {0,1,2}\foreach \J in {3,4,5}
         {\draw[-,line width=1pt,color=gray] (1*\I,1*\J) .. controls +(0.7,0) 
and +(0,0.7) .. (1*\I+3,1*\J-3);}
\foreach \I in {0,...,5}\foreach \J in {0,...,5}
         {\node[draw,shape=circle,fill] (\I,\J) at (1*\I,1*\J) {};}

\end{tikzpicture}}
\end{center}
\caption{The augmented toroidal grid $AT(6,6)$. Gray edges belong to $AT(6,6)$ but not to $T(6,6)$.}
\label{fig:toroidal}
\end{figure}

Observe that after the removal of parallel edges, the graph $AT(2,2)$
is isomorphic to $PC(2)$ (that is $K_4$). Similarly, $AT(4,4)$ is
isomorphic to $PC(4)$. Indeed, $T(4,4)$ is $C_4\Box C_4$ that is $(K_2
\Box K_2) \Box (K_2 \Box K_2)$ which is isomorphic to $H(4)$ (and
antipodal vertices of $T(4,4)$ are antipodal vertices of $H(4)$). The
previous observation is not true for other values of $k$. Nevertheless, in
general, $AT(2k,2k)$ is a subgraph of $PC(2k)$. This property and
others are gathered in the following lemma.

\begin{lemma}
\label{propertiesofAT}
  For any positive integer $k$, let $AT$ denote the graph $AT(2k,2k)$
  and $T$ the graph $T(2k,2k)$. The following statements are true.
  \begin{itemize}
    \item[(i)] $AT$ is a subgraph of $PC(2k)$.
    \item[(ii)] $AT$ is vertex-transitive.
    \item[(iii)] Any two vertices of $AT$ belong to a common
      $(2k+1)$-cycle, hence $AT$ has diameter at most $k$.
    \item[(iv)] $AT$ has odd-girth $2k+1$, hence $AT$ has diameter exactly $k$.
    \item[(v)] Any vertex $v$ in $V(AT)$ can be seen as a vertex in
      $V(T)$ and, for any two vertices $u$ and $v$ in $V(AT)$,
      \begin{equation*}
        d_{AT}(u,v) = \min\{d_T(u,v), 2k+1 - d_T(u,v)\}.
      \end{equation*}
      Moreover, for any vertex $u$ in $V(AT)$ and any integer $d$
      between $1$ and $k$, the neighbourhood of $u$ at distance $d$ in
      $AT$ is the set 
      \begin{equation*}
        N_{AT}^d(u) = N_T^d(u) \cup N_T^{d-1}(\overline{u}).
      \end{equation*}
  \end{itemize}
\end{lemma}

\begin{proof}
  (i). One may label the edges of $AT$ with canonical vectors $e_1, e_2, \ldots, e_m$ of
  $\{0,1\}^{2k}$ and $J$ as follows (indices are now to be understood
  modulo $2k$):

  \smallskip

  $\{(i-1,j),(i,j)\}$ with label $e_{i}$ if $i \leq k$ and $e_{i-k}$
  otherwise,

  \smallskip

  $\{(i,j-1),(i,j)\}$ with label $e_{k+j}$ if $j \leq k$ and $e_{j}$
  otherwise,

  \smallskip

  $\{(i,j),(i+k,j+k)\}$ with label $J$.\\[1mm]
  Note that the binary sum of the labels of the edges along any cycle
  of $AT$ is the all-zero vector. Reciprocally, if the sum of labels
  along a path is the all-zero vector, then this path is closed. Then
  for any path from vertex $(0,0)$ to some vertex $v$ of $AT$, the
  binary sum of the labels is the same. Thus, we may define the
  mapping $\phi$ from the vertices of $AT$ to the vertices of $PC(2k)$
  such that for any vertex $v$ of $AT$, $\phi(v)$ is the binary sum of
  the labels along any path from $(0,0)$ to $v$. The mapping $\phi$ is
  an injective homomorphism from $AT$ to $PC(2k)$. Its image is
  isomorphic to $AT$ which, in turn, is a subgraph of $PC(2k)$.

\medskip

\noindent (ii). Let $v_1$ and $v_2$ be two vertices of $AT$. There are
integers $i_1,i_2,j_1$ and $j_2$ between $0$ and $2k -1$ such that
$v_1=(i_1,j_1)$ and $v_2=(i_2,j_2)$. It is easy to observe that the
mapping $h: (i,j) \mapsto (i+i_2-i_1, j+j_2 -j_1 )$ (operations are modulo $2k$)
is an automorphism of $AT$ mapping $v_1$ to $v_2$.

\medskip

\noindent (iii). Since $AT$ is vertex-transitive, we may assume that
one of these two vertices is the origin $(0,0)$. Let $i$ and $j$ be
two integers between $0$ and $2k-1$. We need to prove that $(0,0)$ and
$(i,j)$ are in a common $(2k+1)$-cycle. By the symmetries of $AT$, we
may assume that $i$ and $j$ are both smaller than or equal to $k$. If
we forget about the antipodal edges, we have the toroidal grid $T$. In
this graph, there is a shortest path from $(0,0)$ to $(k,k)$ going
through $(i,j)$. Together with the antipodal edge $\{(0,0),(k,k)\}$, it
forms a $(2k+1)$-cycle in $AT$ going through $(0,0)$ and $(i,j)$.
  
\medskip

\noindent (iv). This is a consequence of (i), (iii) and the fact that
$PC(2k)$ has odd-girth $2k+1$.

\medskip

\noindent (v). Let $u$ and $v$ be two vertices of $AT$. In (iii), we
described a $(2k+1)$-cycle going through both vertices. This cycle
uses exactly one antipodal edge. Since $AT$ is of odd-girth $2k+1$, by 
Observation~\ref{obs:DistanceDetmByCycle} the
distance between $u$ and $v$ is given by this cycle. On this cycle, we
may distinguish the path from $u$ to $v$ not using the antipodal edge;
it has length $d_T(u,v)$ (see our description in (iii)). The other
path uses an antipodal edge and has length $2k+1 - d_T(u,v)$. The
distance in $AT$ between $u$ and $v$ must be the smaller of these
quantities. Therefore, the first part of (v) is proven.

Let $u$ be a vertex of $AT$. Then, a vertex $v$ can be at distance $d$ from $u$ for two possible reasons. If $d=d_T(u,v)$ then $v$ is in $N_T^d(u)$. Otherwise, $d=2k+1-d_T(u,v)$. In such case, an edge of colour $J$ is used in any shortest path connecting $u$ to $v$. Following the previous proof, one such shortest path can be built starting with the edge $u\overline{u}$ corresponding to $J$. The path from $\overline{u}$ to $v$ is then of length $d-1$ and it is a shortest path connecting the two, thus $v$ is in $N_T^{d-1}(\overline{u})$. Reciprocally, any vertex in $N_T^d(u)$ is in $N_{AT}^d(u)$ because of the distance formula and the fact that $1 \leq d \leq k$. For a vertex $v$ in $N_T^{d-1}(\overline{u})$, any $(2k+1)$-cycle going through $u$, $\overline{u}$ and $v$ uses the edges of a shortest $(d-1)$-path from $\overline{u}$ to $v$. Since $d \leq k$, we may derive that $v$ is in $N_{AT}^d(u)$.
\end{proof}

An illustration of the set $N_{AT}^4((0,0))$ for $k = 7$ described in Lemma~\ref{propertiesofAT}(v), is given in Figure~\ref{fig:N_4(0,0)}. In this figure, towards a simpler presentation, edges connecting top to bottom, right to left and antipodal edges are not depicted.  

\begin{figure}[!ht]
\begin{center}
\scalebox{0.7}{\begin{tikzpicture}[scale=0.6]

\begin{scope}[even odd rule]
\filldraw[draw,thick,rectangle,inner sep=7pt, semitransparent,color=black!25,rotate=45] (7.4,-2.5) rectangle (12.4,2.5) (8.15,-1.75) rectangle (11.65,1.75);
\filldraw[draw,thick,rectangle,inner sep=7pt, semitransparent,color=black!25,rotate=45] (6.7,6.7) rectangle (11.7,7.45);
\filldraw[draw,thick,rectangle,inner sep=7pt, semitransparent,color=black!25,rotate=-45] (6.7,6.7) rectangle (7.45,11.7);
\filldraw[draw,thick,rectangle,inner sep=7pt, semitransparent,color=black!25,rotate=-45] (-3.2,2.45) rectangle (3.2,3.2);
\filldraw[draw,thick,rectangle,inner sep=7pt, semitransparent,color=black!25,rotate=45] (16.6,-1.8) rectangle (17.35,1.8);
\end{scope}

\draw[line width=0.5,color=black!15] (0,0) grid (13,13);
\foreach \I in {0,...,13}\foreach \J in {0,...,13}{
  \node[shape=circle,fill=black,draw=black,minimum size=0.5pt,inner sep=0.5pt](\I\J) at (\I,\J) {};
}
\foreach \I \J in {0/4, 1/3, 2/2, 3/1, 4/0, 0/10, 1/11, 2/12, 3/13, 10/0, 11/1, 12/2, 13/3, 4/7, 5/8, 6/9, 7/10, 8/9, 9/8, 10/7, 9/6, 8/5, 7/4, 6/5, 5/6, 11/13, 12/12, 13/11}{
  \node[shape=circle,fill=black,draw=black,minimum size=0.8pt,inner sep=1.8pt](\I\J) at (\I,\J) {};
}

\foreach \I in {0,...,13}{
  \node[] at (-1,\I) {\I};
  \node[circle,fill=white, inner sep =1,minimum size=0 pt] at (\I,-1) {\I};
}
	
\end{tikzpicture}}
\end{center}
\caption{Illustration of the set $N_{AT}^4((0,0))$ for $k=7$.}
\label{fig:N_4(0,0)}
\end{figure}
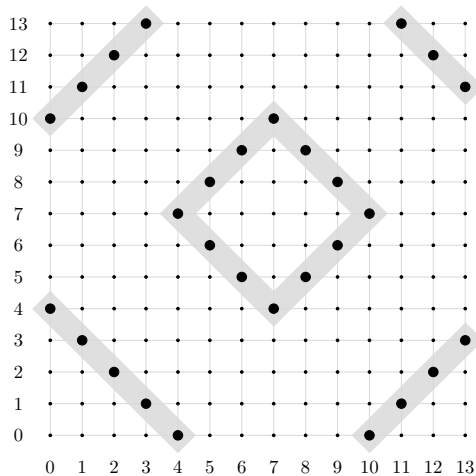

\begin{theorem}\label{thm:ToroidalGrid}
  For any positive integer $k$, the augmented toroidal grid
  $AT(2k,2k)$ bounds $\SP{2k+1}$.
\end{theorem}

\begin{proof}
  Let $k$ be some fixed positive integer and let $AT$ and $T$ denote
  respectively the graphs $AT(2k,2k)$ and $T(2k,2k)$. We shall prove
  that the complete distance graph of $AT$ has the all $k$-good triple
  property. The edge set is clearly non-empty.  Let $p$ be some
  integer between $1$ and $k$ and let $u$ and $v$ be two vertices at
  distance $p$ in $AT$. Let $q$ and $r$ be two integers such that
  $\{p,q,r\}$ is a $k$-good triple. We want to find a vertex $w$ at
  distance $r$ of $u$ and at distance $q$ of $v$. 
  
  Since $AT$ is vertex-transitive, we may assume that $u = (0,0)$. We may also assume, without loss 
of generality, that 
  $q \leq r$. Moreover, thanks to vertical and horizontal symmetries
  of $AT$, we may assume that $v = (a,b)$ where $a$ and $b$ are
  between $0$ and $k$. Finally, since the vertical axis and the
  horizontal axis play the same role, we may assume that $a \leq
  b$. With these assumptions made, we shall restrict ourselves to the
  set of vertices $S^*$ consisting of vertices $(i,j)$ with $i$ and $j$
  between $0$ and $k$ and prove that a suitable vertex $w$ lies in
  $S^*$.
  
  For any integer $d$ between $0$ and $k$, we define $\text{Diag}(d)$ to be
  the set of vertices $(i,j)$ in $S^*$ such that $i+j = d$. By
  Lemma~\ref{propertiesofAT}(v),
  \begin{equation}
    N_{AT}^d(u) \cap S^* = \text{Diag}(d) \cup \text{Diag}(2k+1-d).\label{eqn:N^d(u)}
  \end{equation}
  
  \begin{claim}\label{clm:N^d(i,j)}
    Let $(i,j)$ be a vertex in $S^*$ and $d$ an integer between $1$
    and $k$. Let $(x,y)$ be a vertex in $S^*$.  \\ If $i+j-d\leq x+y\leq
    i+j+d$ and $x-y$ equals $i-j-d$ or $i-j+d$, then $(x,y)$ is in
    $N_{AT}^d((i,j))$.
  \end{claim}
  \claimproof Let $(x,y)$ be a vertex in $S^*$ such that $i+j-d\leq
  x+y\leq i+j+d$ and $x-y=i-j-d$ (the case when $x-y=i-j+d$ is
  analogous). We have $y = x -i+j+d$ then $x+y = 2x -i +j+d$. Since
  $x+y\leq i+j+d$, we deduce that $x \leq i$. Similarly, we may prove
  that $y \geq j$. Thus $\lvert x -i \rvert = i-x$ and $\lvert y-j
  \rvert = y-j$ and $d_T((i,j),(x,y)) = i - x + y-j$. But since
  $x-y=i-j-d$, we have $d_T((i,j),(x,y))=d$. As $d \leq k$,
  Lemma~\ref{propertiesofAT} allows us to conclude that
  $d_{AT}((i,j),(x,y))=d$. This ends the proof of the claim.~\smallqed

\medskip

An illustration of the subset of $N_{AT}^4((2,4))\cap S^*$ described in
Claim~\ref{clm:N^d(i,j)} (for $k = 9$) is shown in Figure~\ref{fig:S*-neighb}.

\begin{figure}[!ht]
\begin{center}
\scalebox{0.8}{\begin{tikzpicture}[scale=0.6]

\begin{scope}[even odd rule]
\filldraw[draw,thick,rectangle,inner sep=7pt, semitransparent,color=black!25,rotate=45] (3.9,3.9) rectangle (7.5,4.55);
\filldraw[draw,thick,rectangle,inner sep=7pt, semitransparent,color=black!25,rotate=-45] (1.05,1) rectangle (1.8,7.5);
\end{scope}

\draw[line width=0.5,color=black!15] (0,0) grid (9,9);
\foreach \I in {0,...,9}\foreach \J in {0,...,9}{
  \node[shape=circle,fill=black,draw=black,minimum size=0.5pt,inner sep=0.5pt](\I\J) at (\I,\J) {};
}
\foreach \I \J in {0/6, 1/7, 2/8, 3/7, 4/6, 5/5, 6/4, 5/3, 4/2, 3/1, 2/0, 1/1, 0/2}{
  \node[shape=circle,fill=black,draw=black,minimum size=0.8pt,inner sep=1.8pt](\I\J) at (\I,\J) {};
}

\node[shape=circle,draw=black,minimum size=0.8pt,inner sep=1.8pt] at (2,4) {};

\foreach \I in {0,...,9}{
  \node[] at (-1,\I) {\I};
  \node[circle,fill=white, inner sep =1,minimum size=0 pt] at (\I,-1) {\I};
}
	
\end{tikzpicture}}
\end{center}
\caption{Illustration of the subset of $N_{AT}^4((2,4))\cap S^*$ (in
  gray) described in Claim~\ref{clm:N^d(i,j)} for $k=9$. The large
  dots are the vertices of $N_{AT}^4((2,4))\cap S^*$.}
\label{fig:S*-neighb}
\end{figure}
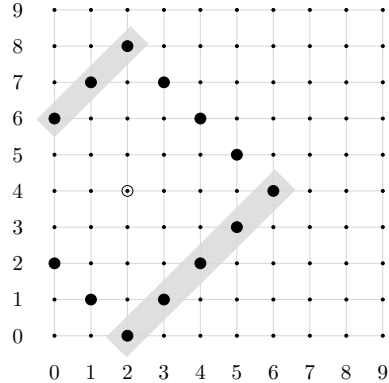

Equation~\eqref{eqn:N^d(u)} tells us that $v$ is either in
$\text{Diag}(p)$ or in $\text{Diag}(2k+1-p)$. In both cases, we may
derive that $p \geq b-a$. Indeed, if $v$ is in $\text{Diag}(p)$, we
have $p = a+b$ and the conclusion is easy. If $v$ is in
$\text{Diag}(2k+1-p)$, we have $p +a +b= 2k+1$ but since $b \leq k$,
we get that $p+a-b \geq 0$. Since $p,q$ and $r$ are distances in the
graph $T_{2k+1}(p,q,r)$ we may use the triangular inequality and
affirm that $r+q \geq p$. In the end, we always have 
\begin{equation} \label{eqn:pba}
  r+q \geq b-a.
\end{equation}

Similarly we show that 
\begin{equation}\label{eqn:rabq}
r - q \leq a+b.
\end{equation} 
Indeed, if $p=a+b$,
it is just the triangular inequality $r \leq p+q$. Otherwise, $a+b =
2k+1-p$ and the quantity $2k+1 -p -r+q$ cannot be negative ($p$ and
$r$ are both at most $k$).

Another inequality we will need is 
\begin{equation}\label{eqn2krabq} 
2k+1 -r \geq a+b-q.
  \end{equation} 
 Once again, if $p=a+b$, it follows easily from the fact that $2k+1 - p-r \geq 0$ (because $p$ and $r$ are at most $k$). If $p=2k+1-a-b$, it is just the expression of the triangular inequality $r \leq p+q$.

To prove our main claim, which is to show the existence of a vertex $w$ at distance $r$ of $u$ and distance $q$ of $v$, we distinguish two cases with respect to the parity of $a+b+q+r$.

\paragraph{Case 1} Suppose that $a+b+q+r$ is even. Then, let $y = \frac{1}{2}(b+r-a-q)$
and $x=r-y=\frac{1}{2}(r+q+a-b)$. First of all, $b+r-a-q$ and
$a+b+q+r$ have the same parity where the latter is assumed to be even. 
Hence $y$ is an integer, and therefore, so is $x$. We have assumed that $a
\leq b$ and $q \leq r$, therefore $y \geq 0$. Since $r$ and $b$ are both
smaller than or equal to $k$, we get that $y \leq k$. Since $y \geq 0$
and $r \leq k$, we observe that $x \leq k$. Inequality~\eqref{eqn:pba}
ensures that $x \geq 0$. We claim that $w=(x,y)$ works. First of all note that $w$ is in
$S^*$. Second, since $x+y=r$, $w$ is in $\text{Diag}(r)$ and by
Equation~\eqref{eqn:N^d(u)}, it is in $N_{AT}^r(u)$. 

It remains to show that $d_{AT}(w,v)=q$.
We have $x-y = a-b+q$ and $x+y=r$. Inequality~\eqref{eqn:rabq} implies that
$x+y \leq a+b+q$. We consider two possibilities based on whether the $p$-path connecting $u$ and 
$v$ uses an edge of colour $J$ or not. Assume $v$ is in $\text{Diag}(p)$, 
in such a case we have $p=a+b$ and the triangular inequality $p \leq q+r$ allows us to derive that 
$a+b - q \leq x+y$. Otherwise $v$ is in $\text{Diag}(2k+1-p)$, and then $a+b = 2k+1-p$.
Hence $p+q+r$ is odd and, therefore, must be greater than or equal to $2k+1$. 
Therefore $a+b-q \leq\ r$. 
In both eventualities, Claim~\ref{clm:N^d(i,j)} ensures that $w$ is also in $N_{AT}^q(v)$.

\paragraph{Case 2} Suppose now that $a+b+q+r$ is odd. Then let $y=
\frac{1}{2}(b-a-q-r+2k+1)$ and $x=2k+1-r-y =
\frac{1}{2}(a-b+q-r+2k+1)$. Then, $b-a-q-r+2k+1$ must be even and $y$
is an integer. In turn, $x$ is also an integer. Since $q$ and $r$ are
at most $k$, we know that $2k+1 -q-r \geq 0$. Since $a \leq b$, we
know that $b-a\geq 0$ so that $y \geq 0$. Moreover, by Inequality~\eqref{eqn:pba}, we get $2k+1 +b-a-q-r \leq 2k+1$ so that $y \leq
k+\frac{1}{2}$. As $y$ is an integer, we have $y \leq k$. Concerning
$x$, the quantity $2k+1 +a+q -b-r$ cannot be negative since $b$ and
$r$ are at most $k$. Thus $x \geq 0$. By assumption, $a \leq b$ and $q
\leq r$ so that $a-b+q-r+2k+1 \leq 2k+1$. Since $x$ is an integer, it
has to be at most $k$. Therefore, if we call $w$ the vertex $(x,y)$,
we have $w$ in $S^*$.\\ Since $x+y=2k+1 -r$, $w$ is in
$\text{Diag}(2k+1-r)$ and by Equation~\eqref{eqn:N^d(u)}, it is in
$N_{AT}^r(u)$.\\ Finally, we have $x-y=a-b+q$ and $x+y = 2k+1-r$. By
Equation~\eqref{eqn2krabq}, we know that $a+b-q \leq x+y$. For the
remaining inequality, we distinguish whether $a+b=p$ or
$a+b=2k+1-p$. If $a+b=p$ then $p+q+r$ is odd and must be at least
$2k+1$ so that $2k+1-r \leq a+b+q$. If $a+b = 2k+1-p$, we just need to
use the triangular inequality $p \leq q+r$ to derive that $2k+1-r \leq
a+b+q$. In both cases, Claim~\ref{clm:N^d(i,j)} ensures that $w$ is
also in $N_{AT}^q(v)$. This completes the proof of Theorem~\ref{thm:ToroidalGrid}.
\end{proof}

\section{Optimal bounds for $\SP{5}$ and $\SP{7}$}\label{sec:smallvalues}

For $k=1$, the triangle ($K_3$) is the best bound one can find for $\SPG=\SP{3}$. This is also best possible in the sense that $K_3\in \SPG$. For $k\geq 2$, in contrast with the case of planar graphs, not only the projective hypercube $PC(2k)$ and the Kneser graph $K(2k+1,k)$ are not optimal bounds, but even the augmented square toroidal grid $AT(2k,2k)$ seems to be non-optimal. In general, we do not know the optimal bounds and leave this an open question, but in this section we describe optimal bounds for the cases $k=2$ and $k=3$. For for $k=1,2,3$, the three optimal bounds also provide the optimal bound for the fractional and circular chromatic numbers of graphs in $\SP{2k+1}$. This suggests that perhaps in general, the optimal bound of odd-girth $2k+1$ for $\SP{2k+1}$ provides, simultaneously, the best bound for both fractional and circular chromatic numbers of graphs in $\SP{2k+1}$. Thus, any such result can be seen as a strengthening of both theorems about these two notions.

\subsection{Bounding $\SP{5}$}

For $k=2$, Theorem~\ref{BoundedByPC(2k)} implies that the Clebsch graph (of order~$16$) bounds $\SP{5}$. By Theorem~\ref{BoundedByK(2k+1,k)}, the Petersen graph (of order~$10$) also bounds $\SP{5}$. One can further check (using Algorithm~\ref{algo}) that some other graphs such as the Dodecahedral graph (of order~$20$), the Armanios-Wells graphs (a distance-regular graph of order~$32$~\cite{Armanios,Wells}) and the Gr\"otzsch graph (of order~$11$) bound $\SP{5}$. From results on circular colouring (see~\cite{HZ00}) we know that $C_{8,3}$ (also known as the Wagner graph) bounds $\SP{5}$. In the next theorem we show that the (unique) optimal bound is obtained from $C_{8,3}$ by removing two edges. More precisely, let $C_8^{++}$ be the graph obtained from an $8$-cycle by adding two disjoint antipodal edges (see Figure~\ref{OptimalSPG5Bound}). Alternatively, $C_8^{++}$  is obtained from the Petersen graph $K(5,2)$ by removing two adjacent vertices. We show next that $C_8^{++}$ is a bound of odd-girth~$5$ for $\SP{5}$ and that it is the unique optimal such bound, both in terms of order and size.

\begin{figure}[ht]
\begin{center}
\scalebox{1.0}{\begin{tikzpicture}[join=bevel,inner sep=0.5mm,scale=0.7]

\node[draw,shape=circle,fill](u0) at (0:2) {};
\node at (0:2.4) {$v_2$};
\node[draw,shape=circle,fill](u1) at (45:2) {};
\node at (45:2.4) {$v_1$};
\node[draw,shape=circle,fill](u2) at (90:2) {};
\node at (90:2.4) {$v_8$};
\node[draw,shape=circle,fill](u3) at (135:2) {};
\node at (135:2.4) {$v_7$};
\node[draw,shape=circle,fill](u4) at (180:2) {};
\node at (180:2.4) {$v_6$};
\node[draw,shape=circle,fill](u5) at (225:2) {};
\node at (225:2.4) {$v_5$};
\node[draw,shape=circle,fill](u6) at (270:2) {};
\node at (270:2.4) {$v_4$};
\node[draw,shape=circle,fill](u7) at (315:2) {};
\node at (315:2.4) {$v_3$};

\draw[color=gray] (u0)--(u2)--(u4)--(u6)--(u0)--(u3)--(u5)--(u7)--(u1)--(u3)--(u6)--(u1)--(u4)--(u7)--(u2)--(u5)--(u0);

\draw[line width=1pt] (u0)--(u1)--(u2)--(u3)--(u4)--(u5)--(u6)--(u7)--(u0) (u1)--(u5) (u3)--(u7);

\end{tikzpicture}}

\end{center}
\caption{The graph $C_8^{++}$ bounding $\SP{5}$ (black edges). The gray edges are the weight~$2$-edges in its corresponding partial distance graph having the all $2$-good triple property.}
\label{OptimalSPG5Bound}
\end{figure}
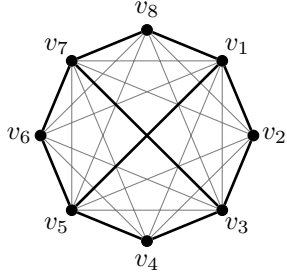

\begin{theorem}\label{C8++}
The graph $C_8^{++}$ bounds $\SP{5}$. Furthermore, this is the unique bound of odd-girth~$5$ for $\SP{5}$ that is optimal both in terms of order and size.
\end{theorem}
\begin{proof}
We will use the labeling of Figure~\ref{OptimalSPG5Bound}. Thus, the vertices of the $8$-cycle are labeled $v_1, v_2, \ldots, v_8$ in the cyclic order. Edges $v_1v_5$ and $v_3v_7$ are the two added antipodal edges. To see that this is a bound, by Theorem~\ref{thm:TripleProperties}, it is sufficient to show that some $2$-partial distance graph of $C_8^{++}$ has the all $2$-good triple property. The $2$-partial distance graph we will consider is the graph which has all edges but the two missing diagonals, that is $K_8-\{v_2v_6, v_4v_8\}$, see Figure~\ref{OptimalSPG5Bound} (black edges have weight~$1$ and gray edges have weight~$2$). 
The list of $2$-good triples is $\{1,1,2\}$, $\{1,2,2\}$, $\{2,2,2\}$. This graph is highly symmetric and using these symmetries, it is enough to check the properties only for the following edges: $v_1v_2$, $v_1v_5$ (both of weight 1) and  $v_1v_3$, $v_1v_4$, $v_2v_4$ (of weight 2). This is a straightforward task.
 
\medskip

Next, we show that any bound of odd-girth~$5$ for $\SP{5}$ must have at least eight vertices. Let $B$ be a minimal bound of odd-girth~$5$. Then $B$ must be a core, has no triangle, and should contain a $5$-cycle. Moreover $B$ cannot be isomorphic to $C_5$ because the graph $T_5(2,2,2)$ from $\SP{5}$ does not admit a homomorphism to $C_5$. It can be easily checked that there are no triangle-free cores on six vertices, so by contradiction we may assume that $B$ has seven vertices. Let $v_1, \ldots, v_5$ be the vertices of a $5$-cycle in $B$ and let $v_6$ and $v_7$ be the other two vertices. Then, since $B$ is triangle-free, $v_6$ is adjacent to at most two vertices of the $5$-cycle. Suppose that it is adjacent to exactly two, without loss of generality it is adjacent to $v_1$ and $v_3$. But then, one of $v_2$ and $v_6$ has its neighbourhood included in the other's (because their only other possible neighbour is $v_7$). Hence identifying $v_2$ and $v_6$ is a homomorphism of $B$ to a proper subgraph of itself, which contradicts with $B$ being a core. Thus, each of $v_6$ and $v_7$ is adjacent to at most one vertex in the $5$-cycle; since they cannot be of degree~$1$, they must be adjacent and then we have two adjacent vertices of degree~$2$ in $B$, contradicting Lemma~\ref{Adjacent2Vertices}. Thus, $B$ has order at least~$8$.
 
Now, we want to prove that $B$ must have at least ten edges. By contradiction, suppose $B$ has at most nine edges. Then, $B$ must have order exactly~$8$, as otherwise either it contains a vertex of degree~$1$ or it contains two adjacent vertices of degree~$2$. Again, we know that $B$ must contain a $5$-cycle and we label the vertices of $B$ with $v_1, \ldots, v_8$ such that $v_1, \ldots, v_5$ induces a $5$-cycle. If for some $i\in\{6,7,8\}$ we have $N(v_i) \subset \{v_1, \ldots, v_5\}$, then $B$ is not a core. Thus, by symmetry, we may assume that $v_6v_7$ and $v_6v_8$ are edges of $B$ (hence $v_7$ is not adjacent to $v_8$ as $B$ is triangle-free).

Based on the adjacencies of $v_6$ in $\{v_1, \ldots, v_5\}$, we consider two cases.

\paragraph{Case 1} Assume that $v_6$ is of degree~$2$. In this case, by Corollary~\ref{Adjacent2Vertices}, $v_7$ and $v_8$ each must have two neighbours among $v_1, \ldots, v_5$, which leaves us with a minimum of eleven edges, a contradiction.

\paragraph{Case 2} Assume that $v_6$ has degree at least~$3$. Without loss of
generality assume that $v_1$ is adjacent to $v_6$. In this case, each of $v_7$ and 
$v_8$ has at least one neighbour in $\{v_1, \ldots, v_5\}$. Thus we 
have a minimum of ten edges. 

Now, if there are exactly ten edges and eight vertices, we also have a
$5$-cycle $v_1\ldots v_5$ and edges $v_6v_7$ and $v_6v_8$. Then each
of $v_6, v_7, v_8$ has exactly one neighbour in $\{v_1, \ldots,
v_5\}$. We claim that neither $v_2$ nor $v_5$ can be such a
neighbour. Otherwise, by symmetries, we may assume that $v_7$ is
adjacent to $v_2$, but identifying $v_7$ and $v_1$ produces a
homomorphism of $B$ to a proper subgraph of itself, which contradicts
with $B$ being a core. Finally, since by Lemma~\ref{Adjacent2Vertices}
there cannot be two adjacent vertices of degree~$2$ in $B$, $v_3$ and
$v_4$ are the neighbours of $v_7$, $v_8$. This graph is isomorphic to
$C_8^{++}$ ($v_4v_5$ and $v_1v_6$ being the two diagonal edges).
\end{proof}

Since $C_8^{++}$ is a proper subgraph of both $C_{8,3}$ and the Petersen graph $K(5,2)$, Theorem~\ref{C8++} is a common strengthening of both the result on the circular chromatic number from~\cite{PZ02a} (which states that all graphs in $\SP{5}$ map to $C_{8,3}$) and the case $k=2$ of Corollary~\ref{cor:fractional} on the fractional chromatic number (that all graphs in $\SP{5}$ map to $K(5,2)$).

\subsection{Bounding $\SP{7}$}

Again, by Theorem~\ref{BoundedByPC(2k)}, Theorem~\ref{BoundedByK(2k+1,k)} and 
Theorem~\ref{thm:ToroidalGrid}, the projective hypercube $PC(6)$ (of order~$64$), the Kneser graph $K(7,3)$ (of order~$35$) and $AT(6,6)$ (of order~$36$) all bound $\SP{7}$. Among other bounds are the Coxeter graph and the $16$-vertex graph $X_{16}$ of Figure~\ref{X16}. The Coxeter graph is a subgraph of $K(7,3)$ of order~$28$, more precisely it is obtained from $K(7,3)$ by removing lines of a Fano plane. More noticeably, it is of girth~$7$. The graph $X_{16}$ is the smallest induced subgraph of $PC(6)$ whose \emph{complete} distance graph has the all $3$-good triple property (this was only verified by a computer check). Next, we introduce a graph $X_{15}$ of odd-girth~$7$ which bounds $\SP{7}$. We then show that $15$ is the smallest order of such a bound.

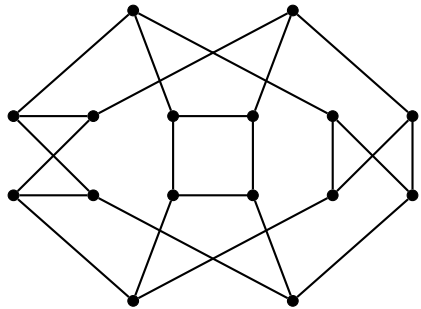
\begin{figure}[ht]
\begin{center}
\scalebox{1}{\begin{tikzpicture}[join=bevel,inner sep=0.5mm,scale=0.7]
\node[draw,shape=circle,fill] (a1) at (2.25,5.5) {};
\node[draw,shape=circle,fill] (a2) at (5.25,5.5) {};

\node[draw,shape=circle,fill] (b1) at (0,3.5) {};
\node[draw,shape=circle,fill] (b2) at (1.5,3.5) {};
\node[draw,shape=circle,fill] (b3) at (3,3.5) {};
\node[draw,shape=circle,fill] (b4) at (4.5,3.5) {};
\node[draw,shape=circle,fill] (b5) at (6,3.5) {};
\node[draw,shape=circle,fill] (b6) at (7.5,3.5) {};

\node[draw,shape=circle,fill] (c1) at (0,2) {};
\node[draw,shape=circle,fill] (c2) at (1.5,2) {};
\node[draw,shape=circle,fill] (c3) at (3,2) {};
\node[draw,shape=circle,fill] (c4) at (4.5,2) {};
\node[draw,shape=circle,fill] (c5) at (6,2) {};
\node[draw,shape=circle,fill] (c6) at (7.5,2) {};

\node[draw,shape=circle,fill] (d1) at (2.25,0) {};
\node[draw,shape=circle,fill] (d2) at (5.25,0) {};

\draw[line width=0.8pt] (b5)--(a1)--(b1)--(b2)--(a2)--(b6);
\draw[line width=0.8pt] (a1)--(b3)--(b4)--(a2);
\draw[line width=0.8pt] (b2)--(c1)--(d1)--(c5)--(b6);
\draw[line width=0.8pt] (d1)--(c3)--(b3);
\draw[line width=0.8pt] (b4)--(c4)--(c3);
\draw[line width=0.8pt] (b1)--(c2)--(c1);
\draw[line width=0.8pt] (c2)--(d2)--(c4);
\draw[line width=0.8pt] (c5)--(b5)--(c6);
\draw[line width=0.8pt] (d2)--(c6)--(b6);
\end{tikzpicture}}
\end{center}
\caption{A $16$-vertex bound for $\SP{7}$.}
\label{X16}
\end{figure}

The graph $X_{15}$ is built from a $10$-cycle whose vertices are labeled (in cyclic order) $v_0,v_1, \ldots, v_9$. Then, a set of five vertices $x_0, x_1, \ldots, x_4$ is added, with $x_i$ being adjacent to $v_i$ and $v_{i+5}$. See Figure~\ref{BoundingSPGgirth7} for two drawings of $X_{15}$. This graph is an induced subgraph of the Kneser graph $K(7,3)$ (and, therefore, also of $PC(6)$), and has circular chromatic number~$5/2$.

\begin{figure}[ht]
\centering
\label{BoundingSPGgirth7-a}
\subfigure[]{\scalebox{1}{\begin{tikzpicture}[join=bevel,inner sep=0.5mm,scale=0.7]
\node[draw,shape=circle,fill] (a) at (0:1.25) {};
\path (a)+(0,0.3) node {$x_0$};
\node[draw,shape=circle,fill] (b) at (72:1.25) {};
\path (b)+(-0.4,0.1) node {$x_1$};
\node[draw,shape=circle,fill] (c) at (144:1.25) {};
\path (c)+(-0.1,-0.4) node {$x_2$};
\node[draw,shape=circle,fill] (d) at (216:1.25) {};
\path (d)+(0.1,-0.4) node {$x_3$};
\node[draw,shape=circle,fill] (e) at (288:1.25) {};
\path (e)+(0.45,0) node {$x_4$};

\node[draw,shape=circle,fill] (0) at (0:2.5) {};
\node at (0:2.9) {$v_0$};
\node[draw,shape=circle,fill] (1) at (36:2.5) {};
\node at (36:2.9) {$v_1$};
\node[draw,shape=circle,fill] (2) at (72:2.5) {};
\node at (72:2.9) {$v_2$};
\node[draw,shape=circle,fill] (3) at (108:2.5) {};
\node at (108:2.9) {$v_3$};
\node[draw,shape=circle,fill] (4) at (144:2.5) {};
\node at (144:2.9) {$v_4$};
\node[draw,shape=circle,fill] (5) at (180:2.5) {};
\node at (180:2.9) {$v_5$};
\node[draw,shape=circle,fill] (6) at (216:2.5) {};
\node at (216:2.9) {$v_6$};
\node[draw,shape=circle,fill] (7) at (252:2.5) {};
\node at (252:2.9) {$v_7$};
\node[draw,shape=circle,fill] (8) at (288:2.5) {};
\node at (288:2.9) {$v_8$};
\node[draw,shape=circle,fill] (9) at (324:2.5) {};
\node at (324:2.9) {$v_9$};

\draw[line width=0.8pt] (0)--(1)--(2)--(3)--(4)--(5)--(6)--(7)--(8)--(9)--(0);
\draw[line width=0.8pt] (0)--(a)--(5);
\draw[line width=0.8pt] (1)--(d)--(6);
\draw[line width=0.8pt] (2)--(b)--(7);
\draw[line width=0.8pt] (3)--(e)--(8);
\draw[line width=0.8pt] (4)--(c)--(9);
\end{tikzpicture}}}\qquad
\subfigure[]{\scalebox{1}{\begin{tikzpicture}[join=bevel,inner sep=0.5mm,scale=0.7]
\node[draw,shape=circle,fill] (0) at (-6:2.5) {};
\node at (-6:2.9) {$v_0$};
\node[draw,shape=circle,fill] (1) at (18:2.5) {};
\node at (18:2.9) {$x_0$};
\node[draw,shape=circle,fill] (2) at (42:2.5) {};
\node at (45:2.9) {$v_5$};
\node[draw,shape=circle,fill] (3) at (66:2.5) {};
\node at (66:2.9) {$v_6$};
\node[draw,shape=circle,fill] (4) at (90:2.5) {};
\node at (90:2.9) {$x_1$};
\node[draw,shape=circle,fill] (5) at (114:2.5) {};
\node at (114:2.9) {$v_1$};
\node[draw,shape=circle,fill] (6) at (138:2.5) {};
\node at (138:2.9) {$v_2$};
\node[draw,shape=circle,fill] (7) at (162:2.5) {};
\node at (162:2.9) {$x_2$};
\node[draw,shape=circle,fill] (8) at (186:2.5) {};
\node at (186:2.9) {$v_7$};
\node[draw,shape=circle,fill] (9) at (210:2.5) {};
\node at (210:2.9) {$v_8$};
\node[draw,shape=circle,fill] (10) at (234:2.5) {};
\node at (234:2.9) {$x_3$};
\node[draw,shape=circle,fill] (11) at (258:2.5) {};
\node at (258:2.9) {$v_3$};
\node[draw,shape=circle,fill] (12) at (280:2.5) {};
\node at (280:2.9) {$v_4$};
\node[draw,shape=circle,fill] (13) at (304:2.5) {};
\node at (304:2.9) {$x_4$};
\node[draw,shape=circle,fill] (14) at (328:2.5) {};
\node at (328:2.9) {$v_9$};

\draw[line width=0.8pt] (0)--(1)--(2)--(3)--(4)--(5)--(6)--(7)--(8)--(9)--(10)--(11)--(12)--(13)--(14)--(0);
\draw[line width=0.8pt] (0)--(5) (3)--(8) (6)--(11) (9)--(14) (12)--(2);

\end{tikzpicture}}}

\caption{Two drawings of the graph $X_{15}$, an optimal $15$-vertex bound for $\SP{7}$.}
\label{BoundingSPGgirth7}
\end{figure}
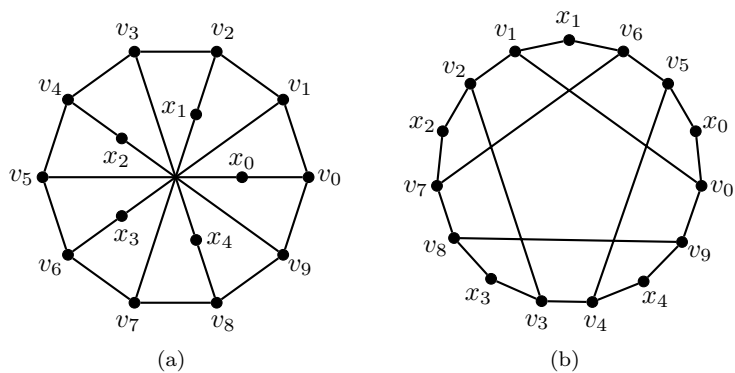

We will show that $X_{15}$ is an optimal bound of odd-girth~$7$ for $\SP{7}$ in terms of the order. In other words, we show no such bound on $14$ or less vertices exists. To prove this, we use the following result of ours on the circular chromatic number of small graphs of odd-girth~$7$~\cite{BFN16}. (Note that the value $14$ is optimal in the statement of Theorem~\ref{thm:OG7-order14-mappingC5}, see~\cite{BFN16} for details.)

\begin{theorem}[\cite{BFN16}]\label{thm:OG7-order14-mappingC5}
Any graph $G$ of order at most $14$ and odd-girth at least~$7$ admits a homomorphism to $C_5$.
\end{theorem}

\begin{theorem}\label{BoundSPG_7}
The graph $X_{15}$ bounds $\SP{7}$. Furthermore, it is a bound of odd-girth~$7$ for $\SP{7}$ that is optimal in terms of the order.
\end{theorem}
\begin{proof} To see that $X_{15}$ bounds $\SP{7}$, observe that the partial distance graph $(\widetilde{X_{15}},d_{X_{15}})$, where $\widetilde{X_{15}}$ contains all pairs of vertices except for the pairs $x_ix_j$, has the all $3$-good triple property. The list of $3$-good triples is $\{1,1,2\}$, $\{1,2,3\}$, $\{1,3,3\}$, $\{2,2,2\}$, $\{2,2,3\}$, $\{2,3,3\}$, $\{3,3,3\}$.
Considering the symmetries of $X_{15}$, it is enough to check the existence of all possible triangles on the following edges: $v_0v_1$, $v_0 x_0$ (of weight~$1$), $v_0v_2$, $v_0x_3$, $v_0v_5$ (weight~$2$), $v_0v_3$, $v_0v_4$, $v_0x_1$,$x_0x_3$ (weight~$3$). Once again this is straightforward. 

It remains to show that no graph of odd-girth~$7$ on $14$ vertices (or less) bounds $\SP{7}$. Towards a contradiction, assume that $B$ is such a bound and moreover assume $B$ is a minimal such bound, that is, no smaller graph (in terms of the order and size) of odd-girth~$7$ bounds $\SP{7}$.

Thus, in particular, $B$ is a core and, therefore, has minimum degree~$2$. Moreover, for any pair $u,v$ of nonadjacent vertices of $B$, there is a $5$-walk connecting $u$ and $v$ (otherwise, identifying $u$ and $v$ and leaving all other vertices untouched produces a bound of smaller order which is also of odd-girth~$7$).

By Lemma~\ref{Degree2Vertex-cycle}, each vertex of degree~$2$ is part of a $6$-cycle and by Lemma~\ref{Adjacent2Vertices}, $B$ has no pair of adjacent vertices of degree~$2$. Furthermore, by Theorem~\ref{thm:circ_chr_nr} and Theorem~\ref{thm:OG7-order14-mappingC5}, the circular chromatic number of $B$ is exactly $5/2$.

\medskip

Let $\psi$ be a $\frac{5}{2}$-circular colouring of $B$, that is, a mapping of $B$ to the $5$-cycle $C_5$ with $V(C_5)=\{0,1,2,3,4\}$ (although this labeling of $C_5$ does not correspond to the definition of a $\frac{5}{2}$-circular colouring, for simplicity we let the edges of $C_5$ be of the form $i(i+1 \bmod 5)$). By Proposition~\ref{prop:tight-cycle} there should be a tight $5k$-cycle in $B$ for some $k\geq 1$; since $B$ has at most $14$ vertices and odd-girth~$7$ we have $k=2$. Let $C$ be this tight $10$-cycle, and label its vertices with  $v_0, v_2, \ldots, v_9$, in cyclic order. Then we can assume that $\psi(v_i)= i \bmod 5$. From this colouring, it is clear that $C$ does not contain any chord, for otherwise there would be a $5$-cycle in $B$.

Let $x$ be a vertex in $V(B)\setminus V(C)$ and assume by symmetry that $\psi(x)=1$. Then, considering the colouring $\psi$, the only possible neighbours of $x$ in $C$ are $v_0, v_2, v_5$ and $v_7$. Vertex $x$ cannot be adjacent to both $v_2$ and $v_5$ as otherwise $B$ has a $5$-cycle. Similarly it cannot be adjacent to both $v_0$ and $v_7$. Next we show that it cannot be adjacent to both $v_0$ and $v_2$. By contradiction suppose $v$ is adjacent to $v_0$ and $v_2$. Then, $x$ cannot be adjacent to $v_1$ and there must be a $5$-walk connecting $x$ and $v_1$. Since there is also a $2$-path connecting them, the $5$-walk must be a $5$-path. Let its vertices be labeled $x,x_0,x_4,x_3,x_2,v_1$ in the order of the path. Furthermore, since its vertices are part of a $7$-cycle and $\psi(x)=\psi(v_1)=1$, $\psi$ must map this path onto $C_5$. By symmetry of $x$ and $v_1$, we may assume $\psi(x_i)=i$. Since any tight $10$-cycle must be chordless, and since replacing $x$ with $v_1$ in $C$ would result in a new tight $10$-cycle, $x_0$ and $x_2$ are distinct from vertices of $C$. But overall, we have at most $14$ vertices, thus one of $x_3$ or $x_4$ must be on $C$. By symmetry of these two vertices (with respect to the currently forced and coloured structure) we may assume $x_3$ is a vertex of $C$. Considering the colouring $\psi$ either we have $x_3=v_3$ in which case $v_3x_4x_0xv_2$ is a $5$-cycle of $B$, or $x_3=v_8$ in which case $v_8v_9v_0v_1x_2$ is a 5-cycle of $B$, a contradiction in both cases. In conclusion we obtain the following claim.

\begin{claim}\label{clm:OnlyDigonaladjacency}
  Any vertex $x$ in $V(B) \setminus V(C)$ is adjacent to at most two vertices of $C$ and if adjacent to two vertices, then those vertices are antipodal in $C$.
\end{claim}


Since there are at most four vertices not in $V(C)$, and since there are five antipodal pairs in $C$, for one such pair, say $v_0$, $v_5$, both vertices are of degree~$2$ in $B$. Thus, by Lemma~\ref{Degree2Vertex-cycle}, $v_0$ belongs to a $6$-cycle; call it $C_{v_0}$. Then, $v_9$ and $v_1$ are necessarily vertices of $C_{v_0}$. Since $C$ is chordless and by Claim~\ref{clm:OnlyDigonaladjacency}, $C_{v_0}$ shares three or four (consecutive) vertices with $C$. 

We first show that they cannot share four vertices. By symmetry,
assume $v_8$ is the fourth vertex in common. Let $u$ and $v$ be the
other two vertices of $C_{v_0}$, assuming $u$ is adjacent to $v_1$. It
follows that $\psi(u)=0$ and $\psi(v)=4$, therefore replacing $v_0$ by
$u$ and $v_9$ by $v$ in $C$ results in another tight $10$-cycle. Since
$u$ is not adjacent to $v_0$, there must be a $5$-walk connecting $u$
and $v_0$, but since $v_0$ and $u$ are at distance~$2$, this $5$-walk
is a $5$-path. It must have $v_9$ as a vertex. Label the other
vertices $x_1$, $x_2$ and $x_3$. It follows
from Claim~\ref{clm:OnlyDigonaladjacency} and the fact that tight $10$-cycles
are chordless that all these three vertices must be new vertices,
contradicting the fact that $B$ has at most $14$ vertices.

Therefore, we may assume that $v_9v_0v_1xyz$ is a $6$-cycle containing
$v_0$ and $x,y,z$ are all distinct from vertices of $C$. Let $t$ be
the last possible vertex (when $|V(B)|=14$). The possible colours for
$y$ are $1$ or $4$, and by symmetry of these two colours we may assume
$\psi(y)=4$. Thus $y$ is not adjacent to $v_2$ (as $\psi(v_2)=2$), an
edge between $y$ and $v_3$ would result in a $5$-cycle, and neither
$x$ nor $z$ can be adjacent to $v_2$ or $v_3$
by Claim~\ref{clm:OnlyDigonaladjacency}. Thus, by
Lemma~\ref{Adjacent2Vertices}, $t$ is adjacent to one of $v_2$ or
$v_3$. On the other hand, since $v_5$ is of degree~$2$ and again by
Lemma~\ref{Adjacent2Vertices}, both $v_4$ and $v_6$ must have
neighbours in $V(B)\setminus V(C)$. By Claim~\ref{clm:OnlyDigonaladjacency}
and the value of $\psi(y)$, the only possibility is that $x$ is
adjacent to $v_6$ and $z$ is adjacent to $v_4$. Now $y$ cannot be
adjacent to $v_7$ or $v_8$ since it would create a $5$-cycle. Thus, by
Lemma~\ref{Adjacent2Vertices} for pairs $v_2,v_3$ and $v_7,v_8$,
by Claim~\ref{clm:OnlyDigonaladjacency}, and by the symmetry of $v_2,v_7$ and
$v_3,v_8$ we may assume $t$ is adjacent to $v_2$ and $v_7$. The only
other edge that can now be added without creating a shorter odd-cycle
or contradicting $\psi$ is the edge $xt$. But then, $v_3$ is a
degree~$2$ vertex which does not belong to any $6$-cycle. This
contradiction completes the proof.
\end{proof}

Again, observe that $X_{15}$ has circular chromatic number $5/2$ and is a proper subgraph of the Kneser graph $K(7,3)$. Hence, Theorem~\ref{BoundSPG_7} can be seen as a common strengthening of both the result on the circular chromatic number of graphs in $\SP{7}$ (from~\cite{HZ00}) and the case $k=3$ of Corollary~\ref{cor:fractional} on the fractional chromatic number of this family of graphs.

\section{Applications to edge-colourings}\label{sec:applis}

In this section, we present edge-colouring results for $K_4$-minor-free multigraphs that follow from our previous results.

Let $G$ be an $r$-regular multigraph. If $G$ contains a set $X$ of an odd number of vertices such that at most $r-1$ edges connect $X$ to $V(G)-X$, then $G$ cannot be $r$-edge-coloured. An $r$-regular multigraph without such a subset of vertices is called an $r$-\emph{graph}. Seymour's result from~\cite{S90} implies in particular that every $K_4$-minor-free $r$-graph is $r$-edge-colourable. Here, we show that for odd values of $r$, this claim is an easy consequence of Theorem~\ref{BoundedByPC(2k)}, thus giving an alternative proof for the odd cases of Seymour's result of~\cite{S90}. Our results in Sections~\ref{sec:general} and~\ref{sec:smallvalues} are therefore strengthenings of this fact, and we present stronger edge-colouring applications of these 
results.

\begin{theorem}\label{thm:dual}
For every $K_4$-minor-free $(2k+1)$-graph $G$, $\chi'(G)=2k+1$.
\end{theorem}
\begin{proof} The proof is the same as a similar proof for planar graphs of~\cite{N07}. We give the main idea. Consider a planar embedding of $G$ and let $G^*$ be the dual of $G$ with respect to this embedding. Note that $G^*$ is also $K_4$-minor-free. Furthermore, it is not difficult to verify (see~\cite{N07} for details) that the hypothesis on $G$ is equivalent to the fact that $G^*$ has odd-girth $2k+1$. Thus, by Theorem~\ref{BoundedByPC(2k)}, $G^*$ admits a homomorphism, say $\phi$, to $PC(2k)$. On the other hand, $PC(2k)$ has a natural $(2k+1)$-edge-colouring by its definition as a Cayley graph. In such a colouring, each $(2k+1)$-cycle receives all $2k+1$ different colours. This edge-colouring of $PC(2k)$ induces an edge-colouring $\psi$ of $G^*$ using the homomorphism $\phi$. While $\psi$ is not necessarily a proper edge-colouring of $G^*$, it inherits the property that each $(2k+1)$-cycle of $G^*$ is coloured with $2k+1$ different colours. Thus, if we colour each edge of $G$ with the colour of its corresponding edge in $G^*$, we obtain a proper edge-colouring.
\end{proof}

Consider the edge-colouring of $C_8^{++}$ induced by $PC(4)$ viewed as the Cayley graph $(\mathbb{Z}_2^{4}, S=\{e_1,e_2,e_3, e_{4},J\})$, where each edge is coloured with the vector of $S$ corresponding to the difference between its endpoints (See Figure~\ref{EdgeColouredC8++} for an illustration).

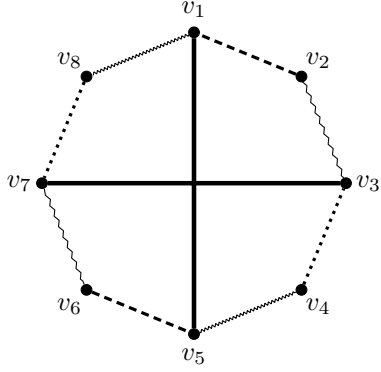
\begin{figure}[ht]
\begin{center}
\scalebox{1.0}{\begin{tikzpicture}[join=bevel,inner sep=0.5mm]

\node[draw,shape=circle,fill](u0) at (0:2) {};
\path (u0)+(0.3,0) node {$v_3$};
\node[draw,shape=circle,fill](u1) at (45:2) {};
\path (u1)+(0.22,0.22) node {$v_2$};
\node[draw,shape=circle,fill](u2) at (90:2) {};
\path (u2)+(0,0.3) node {$v_1$};
\node[draw,shape=circle,fill](u3) at (135:2) {};
\path (u3)+(-0.22,0.22) node {$v_8$};
\node[draw,shape=circle,fill](u4) at (180:2) {};
\path (u4)+(-0.3,0) node {$v_7$};
\node[draw,shape=circle,fill](u5) at (225:2) {};
\path (u5)+(-0.22,-0.22) node {$v_6$};
\node[draw,shape=circle,fill](u6) at (270:2) {};
\path (u6)+(0,-0.3) node {$v_5$};
\node[draw,shape=circle,fill](u7) at (315:2) {};
\path (u7)+(0.22,-0.22) node {$v_4$};

\draw[line width=1.7pt,] (u0)--(u4) (u2)--(u6);

\draw[decoration = {zigzag,segment length = 1.5mm, amplitude = .2mm},decorate] (u0)--(u1) 
(u4)--(u5);

\draw[decoration = {zigzag,segment length = .5mm, amplitude = 
.2mm},decorate] (u2)--(u3) (u6)--(u7); 

\draw[line width=1.2pt, densely dashed] (u1)--(u2) (u5)--(u6);

\draw[line width=1.2pt, dotted] (u3)--(u4) (u7)--(u0);

\end{tikzpicture}}

\end{center}
\caption{The $5$-edge-colouring of $C_8^{++}$ induced by the canonical edge-colouring of $PC(4)$.}
\label{EdgeColouredC8++}
\end{figure}

Note that there are only four $5$-cycles in $C_8^{++}$, and there are only two different cyclic orders of edge-colours induced by these four $5$-cycles. Then, using the technique of the proof of Theorem~\ref{thm:dual}, we can prove the following.

\begin{theorem}\label{thm:edgecolor-perm-C8++}
Let $G$ be a plane $K_4$-minor-free $5$-graph. Let $\{c_1,c_2,c_3,c_4,c_5\}$ be a set of five colours. Then, one can colour the edges of $G$ such that at each vertex, the cyclic order of colours is either $c_1,c_2,c_3,c_4,c_5$ or $c_1,c_4,c_5,c_2,c_3$.
\end{theorem}
\begin{proof}
The proof is the same as the one of Theorem~\ref{thm:dual}, except that using Theorem~\ref{C8++} we can consider a homomorphism $\phi$ of $G$ to $C_8^{++}$ and the $5$-edge-colouring of $C_8^{++}$ induced by $PC(4)$ which is depicted in Figure~\ref{EdgeColouredC8++}.
\end{proof}

As an application, consider a $5$-colourable plane $5$-graph $G$. Furthermore, suppose that the set of colours available to us are green, dark blue and light blue, dark red and light red. Then, one wishes to $5$-edge-colour $G$ so that at each vertex $x$, in a circular ordering of the edges incident to $x$, the colours light blue and dark blue (light red and dark green, respectively) do not appear consecutively. We call a proper edge-colouring satisfying this constraint, a \emph{super proper edge-colouring}. We obtain the following direct consequence of Theorem~\ref{thm:edgecolor-perm-C8++}.

\begin{corollary}\label{SuperProperColoringC8++}
Every plane $K_4$-minor-free $5$-graph admits a super proper $5$-edge-colouring.
\end{corollary}

The next proposition shows that the Icosahedral graph (which has a unique planar embedding, see Figure~\ref{fig:icosahedral}), and which is a $5$-edge-colourable planar $5$-regular graph (see the colouring of Figure~\ref{fig:icosahedral}(b)), admits no super proper $5$-edge-colouring. Therefore Corollary~\ref{SuperProperColoringC8++} cannot be extended to planar $5$-graphs.

\begin{proposition}\label{prop:ico}
The plane Icosahedral graph is not super properly $5$-edge-colourable.
\end{proposition}
\begin{proof}
We use the labeling of Figure~\ref{fig:icosahedral}(a). Assume for a contradiction that we have a super proper $5$-edge-colouring. By the uniqueness of the planar embedding of the Icosahedral graph and by its symmetries, we may assume, without loss of generality, that the edge $xa$ is coloured green. Then, again without loss of generality, we can assume that $xb$ and $xz$ are coloured blue (light or dark) and that $xy$ and $xf$ are coloured red (light or dark). Note that the edges of any triangular face must have a green, a blue and a red edge. Therefore, considering the $xfz$-triangle  $fz$ is green and considering the triangle $xaf$, $af$ is blue. Now, at vertex $f$, the edge $fi$ must be red. And then in the triangle $afi$, $ai$ must be green, a contradiction since $a$ has two incident green edges.
\end{proof}

\begin{figure}[ht]
\begin{center}
\subfigure[A planar embedding.]{\scalebox{0.7}{\begin{tikzpicture}[join=bevel,inner sep=0.5mm]

\node[draw,shape=circle,fill](x) at (90:5) {};
\path (x)+(0,0.3) node {$x$};
\node[draw,shape=circle,fill](y) at (210:5) {};
\path (y)+(-0.3,0) node {$y$};
\node[draw,shape=circle,fill](z) at (330:5) {};
\path (z)+(0.3,0) node {$z$};

\node[draw,shape=circle,fill](a) at (90:1.5) {};
\path (a)+(-0.22,0.22) node {$a$};
\node[draw,shape=circle,fill](b) at (150:1.5) {};
\path (b)+(-0.3,0) node {$b$};
\node[draw,shape=circle,fill](c) at (210:1.5) {};
\path (c)+(-0.22,0.22) node {$c$};
\node[draw,shape=circle,fill](d) at (270:1.5) {};
\path (d)+(0,-0.3) node {$d$};
\node[draw,shape=circle,fill](e) at (330:1.5) {};
\path (e)+(0.22,0.22) node {$e$};
\node[draw,shape=circle,fill](f) at (30:1.5) {};
\path (f)+(0.22,0.22) node {$f$};

\node[draw,shape=circle,fill](g) at (150:0.7) {};
\path (g)+(0,0.3) node {$g$};
\node[draw,shape=circle,fill](h) at (270:0.7) {};
\path (h)+(0,0.4) node {$h$};
\node[draw,shape=circle,fill](i) at (30:0.7) {};
\path (i)+(0.25,-0.1) node {$i$};
                    (h)--(d);

\draw[line width=1pt] (x)--(y)--(z)--(x) 
                        (a)--(b)--(c)--(d)--(e)--(f)--(a)
                        (x)--(b)--(y)--(d)--(z)--(f)--(x)
                        (x)--(a)--(g)--(c)--(y)
                        (b)--(g)--(i)--(f)
                        (a)--(i)--(e)--(z)
                        (g)--(h)--(c)
                        (i)--(h)--(e)
                        (h)--(d);

\end{tikzpicture}}}\hspace{0.5cm}
\subfigure[A $5$-edge-colouring.]{\scalebox{0.7}{\begin{tikzpicture}[join=bevel,inner sep=0.5mm]

\node[draw,shape=circle,fill](x) at (90:5) {};
\node[draw,shape=circle,fill](y) at (210:5) {};
\node[draw,shape=circle,fill](z) at (330:5) {};

\node[draw,shape=circle,fill](a) at (90:1.5) {};
\node[draw,shape=circle,fill](b) at (150:1.5) {};
\node[draw,shape=circle,fill](c) at (210:1.5) {};
\node[draw,shape=circle,fill](d) at (270:1.5) {};
\node[draw,shape=circle,fill](e) at (330:1.5) {};
\node[draw,shape=circle,fill](f) at (30:1.5) {};

\node[draw,shape=circle,fill](g) at (150:0.7) {};
\node[draw,shape=circle,fill](h) at (270:0.7) {};
\node[draw,shape=circle,fill](i) at (30:0.7) {};

\draw[line width=1.5pt,densely dashed] (x)--(y)
                      (b)--(c)
                      (a)--(i)
                      (e)--(f)
                      (d)--(z)
                      (g)--(h);

\draw[line width=1.5pt, decoration = {zigzag,segment length = 1.5mm, amplitude = .2mm},decorate]
                      (x)--(b)
                      (e)--(z)
                      (y)--(d)
                      (a)--(g)
                      (h)--(c)
                      (i)--(f);

\draw[line width=1.5pt, decoration = {zigzag,segment length = .7mm, amplitude = .4mm},decorate]
                      (x)--(a)
                      (b)--(y)
                      (g)--(i)
                      (z)--(f)
                      (c)--(d)
                      (h)--(e);

\draw[line width=1.5pt, dotted]
                      (f)--(x)
                      (y)--(z)
                      (g)--(c)
                      (a)--(b)
                      (d)--(e)
                      (i)--(h);

\draw[line width=2pt,]
                      (z)--(x) 
                      (f)--(a)
                      (c)--(y)
                      (b)--(g)
                      (i)--(e)
                      (h)--(d);

\end{tikzpicture}}}


\end{center}
\caption{The Icosahedral graph and a 5-edge-colouring of it.}
\label{fig:icosahedral}
\end{figure}
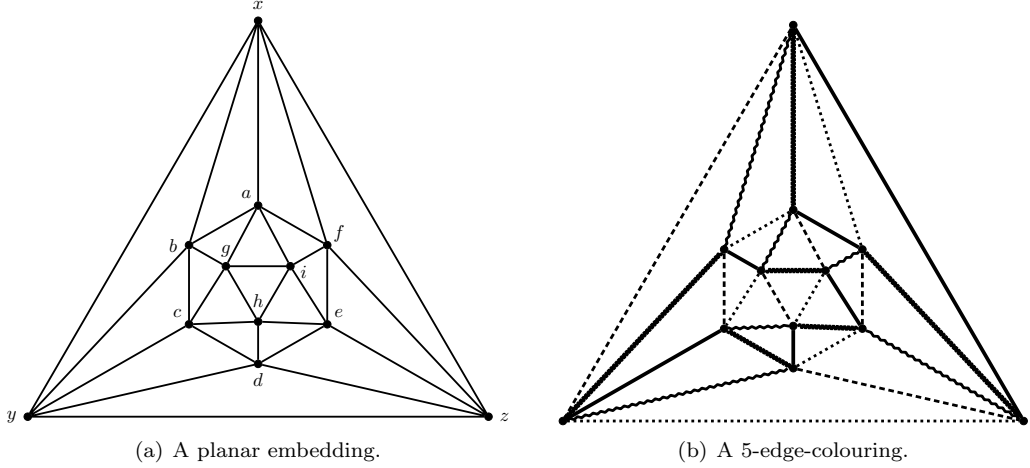

The bound $X_{15}$ of $\SP{7}$ implies a similar edge-colouring result for plane $K_4$-minor-free $7$-graphs as the one of Theorem~\ref{thm:edgecolor-perm-C8++}. To express it, we first describe the $7$-edge-colouring of $X_{15}$ induced by the canonical edge-colouring of $PC(6)$. In this edge-colouring, each edge $v_iv_{i+1}$ of $C_{10}$ is coloured $i \bmod 5$ (thus a total of five colours is used on the edges of $C_{10}$, each twice) Moreover, each pair of edges incident with a vertex $x_j$ ($j=1,2,\ldots, 5$) is coloured with colours $6$ and $7$. In $X_{15}$, any $7$-cycle contains (exactly) one of the $x_j$'s and its two incident edges, and then a continuous set of five edges from $C_{10}$. These five edges therefore always induce the same cyclic order. Thus we obtain the following.

\begin{theorem}
Every plane $K_4$-minor-free $7$-graph $G$ can be edge-coloured with colours $1,2, \ldots, 7$ such that at each vertex, the set of colours $1,2, \ldots, 5$ induces this cyclic order (in either clockwise or anticlockwise order). 
\end{theorem}

Finally, a general result on edge-colouring plane $K_4$-minor-free $(2k+1)$-graphs can be obtained from Theorem~\ref{thm:ToroidalGrid} (that $AT(2k, 2k)$ bounds $\SP{2k+1}$). Note that a $(2k+1)$-cycle of $AT(2k,2k)$ uses exactly $k$ horizontal edges, $k$ vertical edges and an antipodal edge. Furthermore, the set of horizontal edges, in their order of appearance on the cycle, induces a cyclic order of $e_1, e_2, \ldots, e_k$ and similarly the set of vertical edges induces a cyclic order of $e_{k+1}, e_{k+2}, \ldots, e_{2k}$. Thus, we can derive the following definition of special $2k+1$-edge-colourings.

Given $k$, let $B=b_1, b_2,\ldots, b_{k}$ be a sequence of $k$ distinct colours in the family of blue colours and let $R=r_1, r_2, \ldots, r_{k}$ be a sequence of $k$ distinct colours in the family of red colours. Given a $(2k+1)$-regular plane multigraph $G$, we say that $G$ is $(B,R)$-edge-colourable if it can be properly edge-coloured using colours from $B$, $R$ and a unique green colour such that at each vertex $v$, the cyclic ordering of the blue colours (respectively red) around $v$ always induces the same cyclic order as in $B$ (in $R$, respectively).

\begin{theorem}\label{thm:edge-coloring-fromAT}
Let $B$ and $R$ be two sequences of blue and red colours such that $|B|=|R|=k$. Then, every plane $K_4$-minor-free $(2k+1)$-graph is $(B,R)$-edge-colourable. 
\end{theorem}

\section{Concluding remarks}\label{sec:remarks}

We conclude the paper with some remarks and open problems.

\begin{itemize}

\item[1.] As observed before, we have $\chi_c(\SPG)=\chi_c(K_3)$, $\chi_c(\SP{5})=\chi_c(C_8^{++})$, $\chi_c(\SP{7})=\chi_c(X_{15})$ and $\chi_f(\SPG)=\chi_f(K_3)$, $\chi_f(\SP{5})=\chi_f(C_8^{++})$, $\chi_f(\SP{7})=\chi_f(X_{15})$ (see~\cite{HZ00,PZ02a,PZ02b} and Corollary~\ref{cor:fractional} for the values of $\chi_c(\SP{2k+1})$ and $\chi_f(\SP{2k+1})$). We expect that this will generally be the case, that is, the optimal bound of odd-girth~$2k+1$ for $\SP{2k+1}$ should be a subgraph of $PC(2k)$ whose existence improves simultaneously the results on the circular and fractional chromatic numbers of $\SP{2k+1}$.

\item[2.] The \emph{generalized level $k$-Mycielski graph of $C_{2k+1}$}, denoted $M_{k}(C_{2k+1})$, is constructed as follows. We have $V(M_{k}(C_{2k+1}))=V_1\cup\ldots\cup V_{k}\cup\{v\}$, where $V_i=\{u^i_0,\ldots,u^i_{2k}\}$ for $1\leq i\leq k$. The first level, $V_1$, induces a $(2k+1)$-cycle $u^1_0,\ldots,u^1_{2k}$. For each level $V_i$ with $2\leq i\leq k$, vertex $u^i_j$ is adjacent to the two vertices $u^{i-1}_{(j-1)\bmod 2k+1}$ and $u^{i-1}_{(j+1)\bmod 2k+1}$ in the level $V_{i-1}$. Finally, vertex $v$ is adjacent to all vertices in $V_{k}$. Thus, $M_2(C_{5})$ is simply the classic Mycielski construction for $C_5$, that is, the Gr\"otzsch graph. For every $k\geq 1$, $M_{k}(C_{2k+1})$ has odd-girth~$2k+1$, is $4$-chromatic and is a subgraph of $PC(2k)$~\cite{P92}. Note that $M_2(C_{5})$ contains $C_8^{++}$ as a subgraph and therefore bounds $\SP{5}$. Similarly, $M_3(C_{7})$ contains $X_{15}$ as a subgraph and hence bounds $\SP{7}$. Thus, we conjecture that $M_{k}(C_{2k+1})$ bounds $\SP{2k+1}$. (Using Algorithm~\ref{algo}, we have verified this conjecture by a computer check for $k\leq 10$.) Since $M_{k}(C_{2k+1})$ has order $2k^2+k+1$, a confirmation of this conjecture would provide a family of smaller bounds than the augmented square toroidal grids (that have order $4k^2$).

\item[3.] Using the notion of walk-power, it is shown in~\cite{HNS15} that any bound of odd-girth at least $2k+1$ for $\SP{2k+1}$ is of order at least ${k+2 \choose 2}$. Thus, the bounds of Theorem~\ref{thm:ToroidalGrid} are optimal up to a factor of~$8$.

\item[4.] As a strengthening of Proposition~\ref{prop:ico}, does there exist a plane $k$-graph $G$ such that in any $k$-edge-colouring $c$ of $G$ and for any cyclic permutation of the $k$ colours, there is a vertex of $G$ where $c$ induces this permutation of colours?



\end{itemize}
%
%

\end{document}

%% file: OD222.pdf_t
\begin{picture}(0,0)%
\includegraphics{OD222.pdf}%
\end{picture}%
\setlength{\unitlength}{4144sp}%
\begingroup\makeatletter\ifx\SetFigFont\undefined%
\gdef\SetFigFont#1#2#3#4#5{%
  \reset@font\fontsize{#1}{#2pt}%
  \fontfamily{#3}\fontseries{#4}\fontshape{#5}%
  \selectfont}%
\fi\endgroup%
\begin{picture}(3515,3011)(1478,-3219)
\put(4546,-2491){\makebox(0,0)[lb]{\smash{{\SetFigFont{12}{14.4}{\rmdefault}{\mddefault}{\updefault}{\color[rgb]{0,0,0}$w$}%
}}}}
\put(3061,-331){\makebox(0,0)[lb]{\smash{{\SetFigFont{12}{14.4}{\rmdefault}{\mddefault}{\updefault}{\color[rgb]{0,0,0}$v$}%
}}}}
\put(1621,-2446){\makebox(0,0)[lb]{\smash{{\SetFigFont{12}{14.4}{\rmdefault}{\mddefault}{\updefault}{\color[rgb]{0,0,0}$u$}%
}}}}
\end{picture}%